%% file: ex_article.tex
\documentclass[onefignum,onetabnum]{siamonline190516}

\input{ex_shared}

\ifpdf
\hypersetup{
  pdftitle={An Example Article},
  pdfauthor={D. Doe, P. T. Frank, and J. E. Smith}
}
\fi

\newcommand{\RR}{{\mathbb R}}
\newcommand{\bS}{{\mathbb S}}
\newcommand{\mS}{{\mathbf S}}

\newcommand{\N}{\mathbb{N}}
\newcommand{\bx}{\mathbf{x}}
\newcommand{\by}{\mathbf{y}}
\newcommand{\bz}{\mathbf{z}}

\newcommand{\Y}{\mathcal{Y}}
\newcommand{\X}{\mathcal{X}}
\newcommand{\bA}{\mathcal{A}}

\newcommand{\V}{\mathbf{V}}

\newcommand{\bc}{\mathbf{c}}
\newcommand{\be}{\mathbf{e}}
\newcommand{\bw}{\mathbf{w}}
\newcommand{\bu}{\mathbf{u}}
\newcommand{\bss}{\mathbf{s}}
\def\ba{{\boldsymbol{\alpha}}}
\def\bb{{\boldsymbol{\beta}}}
\def\bg{{\boldsymbol{\gamma}}}

\newcommand{\bv}{\mathbf{v}}
\newcommand{\ud}{\mathrm{d}}

\newcommand{\supp}[1]{\mbox{\upshape supp}(#1)}
\newcommand{\tr}[1]{\mbox{\upshape tr}\left(#1\right)}
\newcommand{\trm}[1]{\mbox{\upshape tr}_m\left(#1\right)}
\newcommand{\rank}{\mbox{\upshape rank}}

\newcommand{\psdp}{\mbox{\upshape\tiny primal}}
\newcommand{\dsdp}{\mbox{\upshape\tiny dual}}

\newcommand{\snd}[1]{\vert\N^n_{#1}\vert}

\newcommand{\mH}{\mathscr{H}}
\newcommand{\mL}{\mathscr{L}}

\newcommand{\col}{\text{\upshape col}}

\newcommand{\QM}{\mathcal{Q}}
\newcommand{\ms}{\mathcal{M}}

\newif\ifcomment
\commentfalse
\commenttrue

\begin{document}

\maketitle

\begin{abstract}
We study a class of polynomial optimization problems with a robust polynomial matrix inequality (PMI) constraint where the uncertainty set itself is defined also by a PMI. These can be viewed as matrix generalizations of semi-infinite polynomial programs, since they involve actually infinitely many PMI constraints in general. Under certain SOS-convexity assumptions, we construct a hierarchy of increasingly tight moment-SOS relaxations for solving such problems. Most of the nice features of the moment-SOS hierarchy for the usual polynomial optimization are extended to this more complicated setting. In particular, asymptotic convergence of the hierarchy is guaranteed and finite convergence can be certified if some flat extension condition holds true. To extract global minimizers, we provide a linear algebra procedure for recovering a finitely atomic matrix-valued measure from truncated matrix-valued moments. As an application, we are able to solve the problem of minimizing the smallest eigenvalue of a polynomial matrix subject to a PMI constraint.
If SOS-convexity is replaced by convexity, we can still approximate the optimal value as closely as desired by solving a sequence of semidefinite programs, and certify global optimality in case that certain flat extension conditions hold true. Finally, an extension to the non-convexity setting is provided under a rank one condition. To obtain the above-mentioned results, techniques from real algebraic geometry, matrix-valued measure theory, and convex optimization are employed.
\end{abstract}



\section{Introduction}
Polynomial optimization problems with polynomial matrix inequality (PMI) constraints
have a wide range of applications in many fields \cite{aravanis2022polynomial,HL2006,HL2012,ichihara2009optimal,KN2020,pozdyayev2014atomic}. In particular, as special cases of PMIs, linear or bilinear matrix inequality constrained problems appear frequently in many synthesis problems for linear systems in optimal control \cite{vanantwerp2000tutorial}. 
Due to estimation errors or lack of information, the data of real-world optimization 
problems often involve uncertainty. Hence, robust optimization is an appropriate 
modeling paradigm for some safety-critical applications with little tolerance for failure
\cite{BBC2011}.

In this paper, we study the following robust PMI optimization problem
under data uncertainty in the PMI constraint:
\begin{equation}\label{eq::RCPSO}
\left\{
\begin{aligned}
f^{\star}\coloneqq\inf_{\by\in\Y}&\ f(\by)\\
\text{s.t.}&\ \Y\subseteq\RR^{\ell},\ P(\by, \bx)\succeq 0,\ 
\forall \bx\in \X\subseteq\RR^n,\\
\end{aligned}
\right.\tag{RPMIO}
\end{equation}
where $f(\by)$ is a polynomial function in $\by=(y_1,\ldots,y_\ell)$ which 
is the decision variables constrained in a basic semialgebraic set
\begin{equation}\label{defineY}
\Y\coloneqq\{\by\in\RR^{\ell} \mid \theta_1(\by)\ge 0, \ldots, \theta_s(\by)\ge 0\},
\end{equation}
$\bx=(x_1,\ldots,x_n)$ is the uncertain parameters belonging to some uncertainty set 
\begin{equation}\label{defineX}
    \X\coloneqq\{\bx\in\RR^n\mid G(\bx)\succeq 0\}
\end{equation}
defined by a $q\times q$ symmetric polynomial matrix $G(\bx)$, and $P(\by, \bx)$ is an $m\times m$ symmetric polynomial matrix in $\by$ and $\bx$ (i.e., $P(\by, \bx)$ depends polynomially on the decision variable $\by$ and the uncertain parameter $\bx$). This problem essentially has infinitely many PMI constraints corresponding to different points of $\X$.
We assume that the set of optimizers of \eqref{eq::RCPSO} is nonempty and
make the following sum-of-squares (SOS)-convexity assumptions on \eqref{eq::RCPSO}.
\begin{assumption}\label{assump1}
{\rm
(i) $f(\by),-\theta_1(\by),\ldots,-\theta_s(\by)$ are SOS-convex (Definition \ref{SOS-convex});
(ii) $-P(\by,\bx)$ is PSD-SOS-convex (Definition \ref{psd-sos-convex}) in $\by$ for all $\bx\in\X$;
(iii) $\X$ is compact.
}
\end{assumption}

To highlight the modelling power of \eqref{eq::RCPSO} under Assumption \ref{assump1}, let us name a few problems from different fields which can be modelled as an instance of \eqref{eq::RCPSO}.
First of all, we note that if $\Y=\RR^{\ell}$ and $f(\by)$, $P(\by, \bx)$ depend affinely on $\by$, then we retrieve the robust polynomial semidefinite program (SDP) considered in \cite{SH2006}. 
A basic problem in interval computations is to estimate intervals of confidence for the components of a given vector-valued function when its variables range in a product of intervals. Assume that the function is given by polynomials and we seek an ellipsoid of confidence for its components. Then this problem can be modelled as a robust polynomial SDP (hence an instance of \eqref{eq::RCPSO}).
In optimal control, many problems for systems of ordinary differential equations can be posed as convex optimization problems with matrix inequality constraints which should hold on a prescribed portion of the state space \cite{CG2010,HL2012,SH2006}. If the involved functions in the differential equations are polynomials, these problems often take the form of robust polynomial SDPs. In the context of risk management, the robust correlation stress testing where data uncertainty arises due to untimely recording of portfolio holdings can be formulated as a robust least square SDP fitting into the form of \eqref{eq::RCPSO} \cite{LCP2014}. Moreover, the deterministic PMI optimization problem of maximizing a polynomial function $h(\bx)$ subject to a PMI constraint $G(\bx)\succeq0$ (recall that multiple PMI constraints can be easily merged into one by diagonal augmentation) can be also formulated as an instance of \eqref{eq::RCPSO}:
\begin{equation}\label{eq::DPMI}
    \inf_{y\in\RR}\ y\quad \text{s.t.}\ \ y-h(\bx)\ge0,\ 
\forall \bx\in \X.
\end{equation}
Since the deterministic PMI optimization problem (including the usual polynomial optimization problem as a special case) is already difficult to solve in general, solving \eqref{eq::RCPSO} (even under Assumption \ref{assump1}) is more challenging.

For deterministic PMI optimization problems, Kojima \cite{Kojima2003} proposed SOS
relaxations by utilizing a penalty function and a generalized Lagrangian dual, and subsequently Henrion and Lasserre \cite{HL2006} gave a hierarchy of moment relaxations allowing to detect finite convergence and to extract global minimizers.
Recently, there have been increasing interests in studying robust polynomial optimization 
problems; see e.g. \cite{CJLW2022,JLF2023,JLV2015,LasserreRobust}. However, robust PMI constraints are not considered in these work. For robust polynomial SDPs,
Scherer and Hol \cite{SH2006} provided a hierarchy of matrix SOS relaxations by establishing a Putinar's style Positivstellens\"atz for polynomial matrices. There are also approaches for solving other types of robust SDPs \cite{LCP2014,LB2016,Oishi2006}. As far as the authors know, there is \emph{little} work on how to solve or even approximate \eqref{eq::RCPSO} in general.

Before introducing our main contributions, we would like to point out that the matrix SOS relaxations in \cite{SH2006} cannot be straightforwardly extended to handle \eqref{eq::RCPSO} due to the \emph{nonlinearity} of $P(\by, \bx)$ in $\by$. Also, the dual \emph{moment} facet of the matrix SOS relaxations for \eqref{eq::RCPSO} was not investigated in \cite{SH2006} or elsewhere. This motivates us to establish a moment-SOS hierarchy for \eqref{eq::RCPSO} (under Assumption \ref{assump1}) in full generality.

Let us denote by $\bS[\bx]^{m}$ the cone of $m\times m$ symmetric real polynomial matrices
in $\bx$ and by $\mathcal{P}^m(\X)$ its subcone consisting of polynomial matrices
which are positive semidefinite (PSD) on $\X$. 
By Haviland’s theorem for polynomial matrices (Theorem \ref{th::haviland}), the dual cone 
of $\mathcal{P}^m(\X)$ consists of tracial $\X$-moment functionals on $\bS[\bx]^{m}$
(Definition \ref{def::tf}), while justifying the membership of a linear functional on 
$\bS[\bx]^{m}$ to the dual cone of $\mathcal{P}^m(\X)$ amounts to the matrix-valued
$\X$-moment problem (Definition \ref{def::mmp}). Therefore, to explore the dual aspect of the 
matrix SOS relaxations for \eqref{eq::RCPSO}, we need to invoke the results on
the matrix-valued moment problem. For a given multi-indexed sequence of real
$m\times m$ symmetric matrices $\mathbf{S}=(S_{\ba})_{\ba\in\N^n}$, the matrix-valued
$\X$-moment problem asks if there exists a PSD matrix-valued representing
measure $\Phi$ supported on $\X$ such that $S_{\ba}=\int_{\X} \bx^{\ba}\ud \Phi(\bx)$ for all 
$\ba\in\N^n$. We refer the reader to \cite{KT2022} for a thorough introduction on 
the history and background about the matrix-valued moment problem. For the scalar moment 
problem ($m=1$), due to Haviland’s theorem and Putinar's Positivstellens\"atz, 
the representing measure is guaranteed by the PSDness of the associated moment matrices and localizing matrices. Based on this, Lasserre \cite{LasserreGlobal2001}
proposed the moment-SOS hierarchy for the scalar polynomial optimization and established
its asymptotic convergence. For the truncated scalar moment problem, Curto and Fialkow
\cite{CF1996} gave the celebrated flat extension condition (FEC) on the moment matrix as a sufficient condition for the existence of a representing measure, which allows to detect finite convergence of Lasserre's hierarchy and extract global minimizers \cite{LasserreHenrion}. 
Recently, Kimsey and Trachana \cite{KT2022} obtained a flat extension theorem for the truncated matrix-valued moment problem.
Very recently, Nie \cite[Chapter 10.3.2]{NieBook} investigated the truncated matrix-valued $K$-moment problem where $K$ is defined by polynomial equalities and PMIs.
He gave a solution to this problem using an FEC on the 
{\itshape block rank} of the associated moment matrix. We would like to point 
out that the FEC on the block rank is stronger than on the usual rank (see Remark \ref{rk::FEC}). In the rest of this paper, unless stated otherwise, 
the abbreviation FEC means the flat extension condition on the usual rank of the associated moment matrix.
Unlike the scalar case, to the best of our knowledge, there is \emph{little} work in the literature to link the theory of matrix-valued moments to PMI optimization.
In this paper, we aim to construct a moment-SOS hierarchy for \eqref{eq::RCPSO} with SOS-convexity by employing Scherer-Hol's Positivstellens\"atz and the matrix-valued measure theory.

\vskip 5pt
\noindent{\bf Contributions.} Our main contributions are summarized as follows:
\begin{enumerate}
\item
As the first contribution, we provide a solution to the truncated matrix-valued $\X$-moment problem using the FEC.
Furthermore, we develop a linear algebra procedure for retrieving a finitely atomic representing measure whenever the FEC holds, which extends Henrion-Lasserre's algorithm to the matrix setting.
We remark that those extensions from the scalar setting to the matrix setting are not straightforward applications of existing results. Indeed, the matrix-valued moment matrix is a multi-indexed \emph{block} matrix and the kernel of a truncated PSD moment matrix is a linear subspace rather than a real radical ideal, which make the matrix setting more involved.
\item We establish a moment-SOS hierarchy for solving \eqref{eq::RCPSO} with SOS-convexity.
Here, the main difficulty stems from the nonlinearity of $P(\by, \bx)$ in $\by$. To overcome this obstacle, we reformulate \eqref{eq::RCPSO} to a conic optimization problem via the Lagrange dual theory, and then replace the conic constraints with more tractable matrix quadratic module constraints or matrix-valued pseudo-moment cone constraints.
This yields a sequence of upper bounds on the optimum of \eqref{eq::RCPSO} with guaranteed asymptotic convergence. 
More importantly, we show that if the FEC holds, then finite convergence occurs and we can extract a globally optimal solution $\by^{\star}$ of \eqref{eq::RCPSO} as well as the points $\bx\in\Delta(\by^{\star})$ and the corresponding vectors $\bv$, where 
\begin{equation}\label{eq::deltay}
\Delta(\by^{\star})\coloneqq\{\bx\in\X \mid \exists \bv\in\RR^m\, 
\ { \text{ s.t. } }\ P(\by^{\star}, \bx)\bv=0\}
\end{equation}
is the index set of constraints active at $\by^{\star}$. 
\item
For the linear case of \eqref{eq::RCPSO}, we show that the dual problem is exactly the generalized matrix-valued moment problem and our SOS relaxations recover the matrix SOS relaxations proposed by Scherer and Hol \cite{SH2006}. 
As a complement to the matrix SOS relaxations in \cite{SH2006}, 
the dual matrix-valued moment relaxation allows us to detect finite convergence and to extract optimal solutions.
As an application, we provide a solution to the problem of minimizing the smallest eigenvalue of a polynomial matrix over a set defined by a PMI. 
\item
In case that the SOS-convexity assumption of \eqref{eq::RCPSO} is weakened to convexity, we also provide a sequence of SDPs that can approximate the optimal value of \eqref{eq::RCPSO} as closely as desired. Moreover, finite convergence can be detected via certain FECs. For the general non-convex case, we give respectively conditions under which a lower bound on the optimum, an upper bound on the optimum, or a globally optimal solution can be retrieved.
\end{enumerate}

While the moment-SOS hierarchy provides a powerful framework for scalar polynomial optimization, extending it to handle robust PMI constraints is a non-trivial task even in the SOS-convex case. A straightforward idea would be scalarizing the PMI constraint via the equivalence relation: an $m\times m$ polynomial matrix $P\succeq0$ if and only if $\bz^{\intercal}P\bz\ge0$ for all $\bz\in\RR^m$. However, this approach bears two drawbacks: 1) it necessitates $m$ auxiliary variables which increases the problem size; 2) the information on the structure (e.g., sparsity) of the polynomial matrix becomes less transparent. 
By contrast, our approach natively treats \eqref{eq::RCPSO} at the matrix level, thus preserving its matrix structure and allowing the theory of matrix-valued measures to come into play. The information on matrix structures can be exploited to develop a structured moment-SOS hierarchy in order to tackle large-scale application problems (which will be addressed in a follow-up paper).

\vskip 7pt
The rest of the paper is organized as follows. 
We first recall some preliminaries in Section \ref{sec::pre}.
Then, we consider the truncated matrix-valued $\X$-moment problem in Section \ref{sec::mrecovery} and propose a linear algebra procedure for retrieving the representing measure.
In Section \ref{sec::CRSDP}, we construct a moment-SOS hierarchy for solving \eqref{eq::RCPSO} with SOS-convexity, and treat some special cases. Extensions of the results to the general convex case and the non-convex case are discussed in Section \ref{sec::extension}.
Conclusions are given in Section \ref{sec::conclusions}.

\section{Preliminaries}\label{sec::pre}
We collect some notation and basic concepts
which will be used in this paper. We denote by $\bx$ (resp., $\by$)
the $n$-tuple (resp., $\ell$-tuple) of variables $(x_1,\ldots,x_n)$ (resp.,
$(y_1,\ldots,y_\ell)$).
The symbol $\N$ (resp., $\RR$) denotes
the set of nonnegative integers (resp., real numbers). 
Denote by $\RR^m$ (resp. $\RR^{l_1\times l_2}$, $\bS^{m}$) the $m$-dimensional real vector (resp. $l_1\times l_2$ real matrix, $m\times m$ symmetric real matrix) space. For $\bv\in\RR^m$ (resp., $N\in\RR^{l_1\times l_2}$), the symbol 
$\bv^{\intercal}$ (resp., $N^\intercal$) denotes the transpose of $\bv$ (resp., $N$). For a matrix $N\in\RR^{m\times m}$, $\tr{N}$ denotes its trace. For two matrices $N_1$ and $N_2$, $N_1\otimes N_2$ denotes the 
Kronecker product of $N_1$ and $N_2$.
For two matrices $N_1$ and $N_2$ of the same size, $\langle N_1, N_2\rangle$ denotes the 
inner product $\tr{N_1^{\intercal}N_2}$ of $N_1$ and $N_2$.
The notation $I_m$ denotes the $m\times m$ identity matrix.
For any $t\in \RR$, $\lceil t\rceil$ denotes the smallest
integer that is not smaller than $t$. For $\bu\in \RR^m$,
$\Vert \bu\Vert$ denotes the standard Euclidean norm of $\bu$.
For a vector $\ba=(\alpha_1,\ldots,\alpha_n)\in\N^n$,
let $\vert\ba\vert=\alpha_1+\cdots+\alpha_n$. For a set $A$, we use $\vert A\vert$ to denote its cardinality.
For $k\in\N$, let $\N^n_k\coloneqq\{\ba\in\N^n\mid \vert\ba\vert\le k\}$ and 
$\snd{k}=\binom{n+k}{k}$ be its cardinality.
For variables $\bx \in \RR^n$ and $\ba\in\N^n$, $\bx^{\ba}$ denotes the monomial
$x_1^{\alpha_1}\cdots x_n^{\alpha_n}$.
Let $\RR[\bx]$ (resp. $\bS[\bx]^m$) denote 
the set of real polynomials (resp. $m\times m$ symmetric real polynomial matrices) in $\bx$.
For $h\in\RR[\bx]$, we denote by $\nabla_{\bx}(h)$ its gradient vector and 
by $\nabla_{\bx\bx}(h)$ its Hessian matrix.
For $h\in\RR[\bx]$, we denote by $\deg(h)$ its (total) degree.
For $k\in\N$, denote by $\RR[\bx]_k$ the set of polynomials
in $\RR[\bx]$ of degree up to $k$.
For a polynomial $f(\bx)\in\RR[\bx]$, if there exist polynomials $f_1(\bx),\ldots,f_t(\bx)$ such that $f(\bx)=\sum_{i=1}^tf_i(\bx)^2$,
then we call $f(\bx)$ an SOS polynomial.
For a $\RR$-vector space $A$, denote by $A^*$ the dual space of 
linear functionals from $A$ to $\RR$. Given a cone $B\subseteq A$, 
its dual cone is $B^*\coloneqq\{L\in A^*\mid L(b)\ge 0, \ \forall b\in B\}$.

\subsection{SOS polynomial matrices, SOS-convexity and positivstellensatz for polynomial matrices}

For an $l_1\times l_2$ polynomial matrix $T(\bx)=[T_{ij}(\bx)]$, define
\[
\deg(T)\coloneqq\max\,\{\deg(T_{ij}) \mid i=1,\ldots,l_1,j=1,\ldots,l_2\}.
\]
A polynomial matrix $\Sigma(\bx)\in\bS[\bx]^{q}$ is said to be an \emph{SOS matrix} if there exists an $l\times q$ polynomial matrix $T(\bx)$ for some
$l\in\N$ such that $\Sigma(\bx)=T(\bx)^{\intercal}T(\bx)$. For $d\in\N$, denote by $u_d(\bx)$
the canonical basis of $\RR[\bx]_d$, i.e.,
\begin{equation}\label{eq::ud}
	u_d(\bx)\coloneqq[1,\ x_1,\ x_2,\ \cdots,\ x_n,\ x_1^2,\ x_1x_2,\ \cdots,\
	x_n^d]^{\intercal},
\end{equation}
whose cardinality is $\snd{d}=\binom{n+d}{d}$. 
With $d=\deg(T)$,
we can write $T(\bx)$ as
\[
	T(\bx)=Q(u_d(\bx)\otimes I_q) \text{ with } 
 Q=[Q_1,\ldots,Q_{\snd{d}}],\quad Q_i\in\RR^{l\times q},
\]
where $Q$ is the vector of coefficient matrices of $T(\bx)$ with respect to
$u_d(\bx)$. Hence, $\Sigma(\bx)$ is an SOS matrix with respect to $u_d(\bx)$ if there
exists some $Q\in\RR^{l\times q\snd{d}}$ satisfying 
\[
	\Sigma(\bx)=T(\bx)^{\intercal}T(\bx)=(u_d(\bx)\otimes I_q)^{\intercal}(Q^{\intercal}Q)(u_d(\bx)\otimes I_q). 
\]
We thus have the following result.
\begin{proposition}{\upshape \cite[Lemma 1]{SH2006}}\label{prop::SOSrep}
A polynomial matrix $\Sigma(\bx)\in\bS[\bx]^{q}$ is an SOS matrix with
respect to the monomial basis $u_d(\bx)$ if and only if there exists $Z\in\mathbb{S}_+^{q\snd{d}}$ such that  
\[\Sigma(\bx)=(u_d(\bx)\otimes I_q)^{\intercal} Z (u_d(\bx)\otimes I_q).\]
\end{proposition}

Now let us recall some basic concepts about SOS-convexity.
\begin{definition}{\upshape \cite{SRCSHN}}\label{SOS-convex}
A polynomial $h\in\RR[\by]$ is \emph{SOS-convex} if its Hessian $\nabla_{\by\by} h (\by)$ is an SOS matrix.
\end{definition}

While checking the convexity of a polynomial 
is generally NP-hard \cite{Ahmadi2013}, the SOS-convexity can be justified numerically by solving an SDP by Proposition \ref{prop::SOSrep}. Recall that

\begin{definition}
A polynomial matrix $Q(\by)\in\bS[\by]^m$ is 
\emph{PSD-convex} if 
\[
t Q(\by^{(1)}) + (1-t) Q(\by^{(2)}) \succeq  Q(t\by^{(1)}+(1-t)\by^{(2)})
\]
holds for any $\by^{(1)}, \by^{(2)}\in\RR^{\ell}$ and $t\in (0, 1)$.
\end{definition}

Nie \cite{N2011} gave an extension of SOS-convexity to polynomial matrices.
\begin{definition}{\upshape \cite{N2011}}\label{psd-sos-convex}
A polynomial matrix $Q(\by)\in\bS[\by]^m$ is 
\emph{PSD-SOS-convex}
if for every $\bv\in\RR^m$, there exists a polynomial matrix
$F_{\bv}(\by)$ in $\by$ such that
\[
\nabla_{\by\by}(\bv^{\intercal}Q(\by)\bv)=F_{\bv}(\by)^{\intercal}F_{\bv}(\by).
\]
\end{definition}

In other words, $Q(\by)$ is PSD-SOS-convex if and 
only if $\bv^{\intercal}Q(\by)\bv$ is an SOS-convex polynomial for each $\bv\in\RR^m$.
Clearly, if $Q(\by)$ is PSD-SOS-convex, then it is PSD-convex, 
but not vice versa. The PSD-SOS-convexity condition requires checking the Hessian $\nabla_{\by\by}(\bv^{\intercal}Q(\by)\bv)$ for every
$\bv\in\RR^m$, which is hard in general. See Appendix \ref{secA0} for a stronger and easier-to-check condition called uniform PSD-SOS-convexity given in \cite{N2011}.

\vskip 5pt

We next recall the Positivstellens\"atz for polynomial matrices obtained in \cite{SH2006}.
Define the bilinear mapping 
\[
(\,\cdot\,,\,\,\cdot\,)_m \colon \RR^{mq\times mq}\times \RR^{q\times q} \to \RR^{m\times m},\quad
(A, B)_m=\trm{A^{\intercal}(I_m\otimes B)}, 
\]
with
\[
\trm{C}\coloneqq\left[
\begin{array}{ccc}
\tr{C_{11}} & \cdots & \tr{C_{1m}}\\
\vdots & \ddots & \vdots\\
\tr{C_{m1}} & \cdots & \tr{C_{mm}}
\end{array}
\right]
\quad\text{for } C\in\RR^{mq\times mq}, C_{jk}\in\RR^{q\times q}.
\]

\begin{assumption}\label{assump2}
{\rm
For the defining matrix $G(\bx)$ of $\X$ in \eqref{defineX}, there exists $r\in\RR$ and an SOS polynomial matrix 
$\Sigma(\bx)\in\bS[\bx]^{q}$ such that $r^2 - \Vert \bx\Vert^2 - \langle \Sigma(\bx), G(\bx)\rangle$ is an SOS.}
\end{assumption}
%

\begin{theorem}{\upshape \cite[Corollary 1]{SH2006}}\label{th::psatz}
Let Assumption \ref{assump2} hold and $F(\bx)\in\bS[\bx]^{m}$ be positive definite on $\X$.
Then there exist SOS polynomial matrices $\Sigma_0(\bx)\in\bS[\bx]^{m}$ and 
$\Sigma_1(\bx)\in\bS[\bx]^{mq}$ such that 
\[F(\bx)=\Sigma_0(\bx)+(\Sigma_1(\bx), G(\bx))_m.\]
\end{theorem}

For each $k\in\N$, we define the $k$-th \emph{truncated matrix quadratic module} $\QM^m_k(G)$ associated with $G(\bx)$ by 
\[
\QM^m_k(G)\coloneqq\left\{\Sigma_0(\bx)+(\Sigma_1(\bx), G(\bx))_m\ \middle\vert \
\begin{aligned}
&\Sigma_0\in\bS[\bx]^{m}, \Sigma_1\in\bS[\bx]^{mq},\\
&\Sigma_0, \Sigma_1\ \text{are SOS matrices},\\
&\deg(\Sigma_0), \deg((\Sigma_1, G)_m)\le 2k
\end{aligned}
\right\},
\]
and define the \emph{matrix quadratic module} by $\QM^m(G)\coloneqq\bigcup_{k\in\N}\QM^m_k(G).$
By Proposition \ref{prop::SOSrep}, checking membership in $\QM^m_k(G)$ can be accomplished with an SDP.

\subsection{Matrix-valued measures}
Now we recall some background on the theory of matrix-valued measures, which is crucial for our subsequent development.
For more details, the reader is referred to \cite{DPS2008,DS2003,DL1997,GNR2019}.
Denote by $B(\X)$ the smallest $\sigma$-algebra generated 
from the open subsets of $\X$ and by $\mathfrak{m}(\X)$ the set of all finite Borel measures on $\X$. A measure $\phi\in\mathfrak{m}(\X)$ is \emph{positive} if $\phi(\bA)\ge 0$ for all $\bA\in B(\X)$. Denote by $\mathfrak{m}_+(\X)$ the set of all finite positive Borel measures on $\X$. 
The support $\supp{\phi}$ of a Borel measure $\phi\in\mathfrak{m}(\X)$ is the (unique) smallest closed set $\bA\in B(\X)$ 
such that $\phi(\X\setminus \bA)=0$.

\begin{definition}
Let $\phi_{ij}\in\mathfrak{m}(\X)$, $i, j = 1, \ldots, m$. 
The $m\times m$ matrix-valued measure $\Phi$ on $\X$ is defined as the matrix-valued function 
$\Phi\colon B(\X) \to \RR^{m\times m}$ with
\[\Phi(\bA)\coloneqq[\phi_{ij}(\bA)]\in \RR^{m\times m}, \quad \forall \bA\in B(\X). \]
If $\phi_{ij}=\phi_{ji}$ for all $i, j =1, \ldots, m$, we call $\Phi$ a symmetric matrix-valued measure.
If $\bv^{\intercal} \Phi(\bA) \bv\ge 0$ holds for all $\bA\in B(\X)$ and for all column vectors $\bv\in\RR^m$, we call $\Phi$ a PSD matrix-valued measure. The set $\supp{\Phi}\coloneqq\bigcup_{i,j=1}^m \supp{\phi_{ij}}$
is called the support of the matrix-valued measure $\Phi$. 
A function $h \colon \X \to \RR$ is called $\Phi$-measurable if $h$ is $\phi_{ij}$-measurable for every $i,j = 1, \ldots, m$. The matrix-valued integral of $h$ with respect to the measure $\Phi$ is defined by 
\[\int_{\X} h(\bx) \ud \Phi(\bx)\coloneqq\left[\int_{\X}h(\bx)\ud \phi_{ij}(\bx)\right]_{i,j=1,\ldots,m}\in\RR^{m\times m}.\]
\end{definition}

We denote by $\mathfrak{M}^m(\X)$ (resp. $\mathfrak{M}^m_+(\X)$) the set of all $m\times m$ (resp. PSD) symmetric matrix-valued measures on $\X$. 


\begin{definition}
A \emph{finitely atomic PSD} matrix-valued measure $\Phi\in\mathfrak{M}^m_+(\X)$ is a matrix-valued measure of form $\Phi=\sum_{i=1}^r W_i \delta_{\bx^{(i)}}$ 
where $W_i\in\bS_+^m$, $\bx^{(i)}$'s are distinct points in $\X$, and $\delta_{\bx^{(i)}}$ denotes the Dirac measure centered at $\bx^{(i)}$, $i=1,\ldots,r$.
\end{definition}

\begin{definition}\label{def::tf}
A linear functional $\mathscr{L} \colon \bS[\bx]^{m} \to \RR$ is called a tracial $\X$-moment functional if there exists a matrix-valued measure $\Phi\in \mathfrak{M}_+^m(\X)$ such that 
\begin{equation}\label{eq::L}
\supp{\Phi}\subseteq \X,\ 
\mathscr{L}(F)=\int_{\X}\ \tr{F(\bx)\ud \Phi(\bx)}
 =\sum_{i,j}\int_{\X}\ F_{ij}(\bx)\ud \phi_{ij}(\bx),
\ \ \forall F(\bx)\in\bS[\bx]^{m}.
\end{equation}
The matrix-valued measure $\Phi\in\mathfrak{M}_+^m(\X)$ is called a \emph{representing measure} of $\mathscr{L}$ and we write $\mathscr{L}_{\Phi}$ for $\mathscr{L}$ to indicate the associated measure. 
\end{definition}

Now we define the convex cones
\begin{equation}\label{eq::lm}
\mathcal{L}^m(\X)\coloneqq\{\mathscr{L} \colon \bS[\bx]^{m} \to \RR
\mid \mathscr{L}\ \text{is a tracial $\X$-moment functional}\},
\end{equation}
and 
\begin{equation}\label{eq::pm}
\mathcal{P}^m(\X)\coloneqq\left\{F(\bx)\in\bS[\bx]^{m} \mid F(\bx)\succeq 0, \,\forall \bx\in \X\right\}.
\end{equation}

The following theorem is a matrix version of Haviland’s theorem (\cite{CIMPRIC2013,SCH1987}).
\begin{theorem}[Haviland’s theorem for polynomial matrices]{\upshape \cite[Theorem 3]{CIMPRIC2013}}\label{th::haviland}
A linear functional $\mathscr{L}$ is a tracial $\X$-moment functional
if and only if 
$\mathscr{L}(F)\ge 0$ for all $F(\bx)\in\mathcal{P}^m(\X)$. 
\end{theorem}

\begin{proposition}{\upshape \cite[Proposition 10.3.3]{NieBook}}
The cones $\mathcal{L}^m(\X)$ and $\mathcal{P}^m(\X)$ are dual to each other, i.e., $\mathcal{L}^m(\X)=\mathcal{P}^m(\X)^*$ and $\mathcal{P}^m(\X)={\mathcal{L}^m(\X)}^*$.
\end{proposition}

\subsection{The matrix-valued $\X$-moment problem}
Let $\mathbf{S}=(S_{\ba})_{\ba\in\N^n}$ be a multi-indexed sequence of symmetric
matrices in $\bS^{m}$.
\begin{definition}\label{def::mmp}{\upshape \cite{KimseyPHD}}
For a non-empty closed set $\X\subseteq\RR^n$, the sequence 
$\mathbf{S}=(S_{\ba})_{\ba\in\N^n}\subseteq\bS^{m}$ is called a matrix-valued $\X$-moment sequence if there exists a matrix-valued measure 
$\Phi=[\phi_{ij}]\in\mathfrak{M}^m_+(\X)$ such 
that 
\begin{equation}\label{eq::S}
\supp{\Phi}\subseteq \X\quad\text{and}\quad 
S_{\ba}=\int_{\X} \bx^{\ba}\ud \Phi(\bx),\ \forall \ba\in\N^n.
\end{equation}
The measure $\Phi\in\mathfrak{M}^m_+(\X)$ satisfying \eqref{eq::S} is 
called a \emph{representing measure} of $\mS$.
\end{definition}

For a given sequence $\mS=(S_{\ba})_{\ba\in\N^n}\subseteq\bS^{m}$, we 
can define a linear functional $\mL_{\mS}  \colon \bS[\bx]^{m} \to \RR$ in the 
following way:
\[
\mL_{\mS}(F)\coloneqq\sum_{\ba\in\supp{F}}\tr{F_{\ba}S_{\ba}},\ \forall 
F(\bx)=\sum_{\ba\in\supp{F}} F_{\ba}\bx^{\ba}\in\bS[\bx]^{m},
\]
where $F_{\ba}$ is the coefficient matrix of $\bx^{\ba}$ in $F(\bx)$ and 
\[\supp{F}\coloneqq\{\ba\in\N^n \mid \bx^{\ba} \text{ appears in some } F_{ij}(\bx)\}.\]
We call $\mL_{\mS}$ the \emph{Riesz functional} associated to the sequence $\mS$.
Clearly, $\mS$ is a matrix-valued $\X$-moment sequence if and only if $\mL_{\mS}$ is a tracial 
$\X$-moment functional.

\begin{definition}\label{def::mm}
Given a sequence $\mS=(S_{\ba})_{\ba\in\N^n}\subseteq\bS^{m}$,
the associated \emph{moment matrix} $M(\mS)$ is the block matrix whose block row and block column are indexed by $\N^n$ and the $(\ba, \bb)$-th block entry is $S_{\ba+\bb}$ for all $\ba, \bb\in\N^n$. For $G\in\bS[\bx]^{q}$,
the \emph{localizing matrix} $M(G\mS)$ associated to $\mS$ and $G$ is the block matrix whose block row and block column are indexed by $\N^n$ and 
the $(\ba, \bb)$-th block entry is 
$\sum_{\bg\in\supp{G}}S_{\ba+\bb+\bg}\otimes G_{\bg}$ for all $\ba, \bb\in\N^n$. For $d\in\N$, the $d$-th order moment matrix $M_d(\mS)$ (resp. localizing matrix
$M_d(G\mS)$) is the submatrix of $M(\mS)$ (resp. $M(G\mS)$) whose block row and block column are both indexed by $\N^n_d$.
\end{definition}

The following proposition can be easily verified from the definitions.
\begin{proposition}\label{prop::Lrepre}
Let $\Sigma_0(\bx)\in\bS[\bx]^{m}$ and $\Sigma_1(\bx)\in\bS[\bx]^{mq}$ be SOS matrices such that
\[
\Sigma_0(\bx)=(u_d(\bx)\otimes I_m)^{\intercal} Z_0 (u_d(\bx)\otimes I_m) 
\text{ and } \Sigma_1(\bx)=(u_d(\bx)\otimes I_{mq})^{\intercal} Z_1 (u_d(\bx)\otimes I_{mq}),
\]
with $Z_0\in\bS_+^{m\snd{d}}$ and $Z_1\in\bS_+^{mq\snd{d}}$. Then for a sequence $\mS=(S_{\ba})_{\ba\in\N^n}\subseteq\bS^{m}$, it holds that
\[
\mL_{\mS}(\Sigma_0)=\tr{Z_0M_d(\mS)}\text{ and }
\mL_{\mS}((\Sigma_1, G)_m)=\tr{Z_1M_d(G\mS)}.
\]
\end{proposition}

Let $d_G\coloneqq\lceil\deg(G)/2\rceil$. For each integer $k\ge d_G$, we define the set
\[
\ms_k^m(G)\coloneqq\{\mS=(S_{\ba})_{\ba\in\N^n_{2k}}\subseteq\bS^{m} \mid 
M_k(\mS)\succeq 0,\ M_{k-d_G}(G\mS)\succeq 0\},
\]
and let $\ms^m(G)\coloneqq\bigcap_{k\ge d_G}\ms_k^m(G)$,
which are all convex cones. Checking membership in $\ms_k^m(G)$ can be accomplished with an SDP. 
Moreover, by Proposition \ref{prop::Lrepre}, $\ms_k^m(G)$ is the dual cone of $\QM^m_k(G)$.

For a given $\mS=(S_{\ba})_{\ba\in\N^n}\subseteq\bS^{m}$, the matrix-valued $\X$-moment problem 
asks when there exists a matrix-valued measure $\Phi\in\mathfrak{M}^m_+(\X)$ satisfying 
\eqref{eq::S}. A necessary condition can be derived from Proposition \ref{prop::Lrepre}.
\begin{corollary}\label{cor::nece}
    If $\mS=(S_{\ba})_{\ba\in\N^n}\subseteq\bS^{m}$ has a matrix-valued representing measure 
    $\Phi\in\mathfrak{M}^m_+(\X)$, then $\mS\in\ms^m(G).$ 
\end{corollary}
\begin{proof}
    For any $k\in\N$, PSD matrices $Z_0\in\bS_+^{m\snd{k}}$, and $Z_1\in\bS_+^{mq\snd{k}}$,
    let $\Sigma_0$ and $\Sigma_1$ be the SOS polynomial matrices defined in Proposition \ref{prop::Lrepre}.
    Then, for any $\bx\in\X$, we have $\Sigma_0(\bx)\succeq 0$ and $(\Sigma_1(\bx), G(\bx))_m\succeq 0$ 
    (see \cite{SH2006}). Hence, 
    \[
    \tr{Z_0M_k(\mS)}=\mL_{\mS}(\Sigma_0)=\int_{\X}\ \tr{\Sigma_0(\bx)\ud \Phi(\bx)} \ge 0
    \]
    and 
    \[
    \tr{Z_1M_k(G\mS)}=\mL_{\mS}((\Sigma_1, G)_m)=\int_{\X}\ 
    \tr{(\Sigma_1(\bx), G(\bx))_m\ud \Phi(\bx)} \ge 0.
    \]
    As $Z_0$ and $Z_1$ are arbitrary, we have $M_k(\mS)\succeq 0$ and $M_k(G\mS)\succeq 0$.
\end{proof} 

Moreover, the matrix-valued $\X$-moment problem is addressed in the following theorem.
\begin{theorem}\label{th::MM}{\upshape \cite[Theorems 5 and 6]{CIMPRIC2013}}
Let Assumption \ref{assump2} hold. Given a sequence 
$\mS=(S_{\ba})_{\ba\in\N^n}\subseteq\bS^{m}$, $\mS$ is a matrix-valued $\X$-moment sequence if and only if $\mS\in\ms^m(G)$.
\end{theorem}

\begin{remark}\label{rk::putinar}
If $m=1$, we use the notation $\QM(G)$ (resp., $\QM_k(G)$, $\ms(G)$, $\ms_k(G)$) instead of $\QM^1(G)$ (resp., $\QM_k^1(G)$, $\ms^1(G)$, $\ms_k^1(G)$) for simplicity.
For a set of polynomials $H(\bx)=\{h_1(\bx),\ldots,h_s(\bx)\}\subseteq\RR[\bx]$, by slightly abusing notation, we use $\QM^m(H)$, $\QM^m_k(H)$, $\ms^m(H)$, $\ms_k^m(H)$ to denote the related sets associated with the diagonal matrix $\diag(h_1(\bx),\ldots,h_s(\bx))$.
Then, when $m=1$, Theorems \ref{th::psatz} and \ref{th::MM} recover Putinar's Positivstellens\"atz \cite{Putinar1993} and its dual aspect for the basic semi-algebraic set $\{\bx\in\RR^n \mid h_1(\bx)\ge 0,\ldots,h_s(\bx)\ge 0\}$.
\end{remark}

\section{The FEC and matrix-valued measure recovery}\label{sec::mrecovery}
\subsection{The truncated matrix-valued $\X$-moment problem}
Recently, Kimsey and Trachana \cite{KT2022} obtained a flat extension theorem 
which provides a solution to the truncated matrix-valued moment problem. \begin{theorem}\label{th::FE}{\upshape\cite[Theorem 6.2]{KT2022} (flat extension)}
For a truncated sequence $\mS=(S_{\ba})_{\ba\in\N^n_{2k}}\subseteq\bS^{m}$,
the following statements are equivalent:
\begin{enumerate}
\item[(i)] $\mS$ admits an atomic representing measure $\Phi=\sum_{i=1}^r W_i\delta_{\bx^{(i)}}$
with $W_i\in\bS_+^{m}$, $\bx^{(i)}\in\RR^n$ and $\sum_{i=1}^r\rank(W_i)=\rank(M_k(\mS))$;
\item[(ii)] $M_k(\mS)\succeq 0$ and $\mS$ admits an extension $\tilde{\mS}=(\tilde{S}_{\ba})_{\ba\in\N^n_{2k+2}}$
such that $M_{k+1}(\tilde{\mS})\succeq 0$ and $\rank(M_k(\mS))=\rank(M_{k+1}(\tilde{\mS}))$.
\end{enumerate}
\end{theorem}
When $m=1$, Theorem \ref{th::FE} recovers the celebrated flat extension theorem 
of Curto and Fialkow \cite{CF1996}. 
There is also a version of the result by Curto and Fialkow \cite{CF2000}, which 
characterizes a truncated real sequence having a representing measure supported on 
a prescribed semi-algebraic subset of $\RR^n$. 
We next extend Theorem \ref{th::FE} to matrix-valued measures supported on $\X$, which provides a solution to the truncated matrix-valued $\X$-moment problem.
\begin{theorem}\label{th::FEC}
Given a truncated sequence $\mS=(S_{\ba})_{\ba\in\N^n_{2k}}\subseteq\bS^{m}$, 
the following statements are equivalent:
\begin{enumerate}
\item[(i)] $\mS$ admits an atomic representing measure $\Phi=\sum_{i=1}^r W_i\delta_{\bx^{(i)}}$
with $W_i\in\bS_+^{m}$, $\bx^{(i)}\in\X$ and $\sum_{i=1}^r\rank(W_i)=\rank(M_k(\mS))$;
\item[(ii)] $M_k(\mS)\succeq 0$ and $\mS$ admits an extension $\tilde{\mS}=(\tilde{S}_{\ba})_{\ba\in\N^n_{2(k+d_G)}}$
such that $M_{k+d_G}(\tilde{\mS})\succeq 0$, $M_k(G\tilde{\mS})\succeq 0$ and $\rank(M_k(\mS))=\rank(M_{k+d_G}(\tilde{\mS}))$.
\end{enumerate}
\end{theorem}
\begin{proof}
(i)$\Rightarrow$(ii). It is implied by Corollary \ref{cor::nece} and Theorem \ref{th::FE}.

(ii)$\Rightarrow$(i). By Theorem \ref{th::FE}, $\mS$ admits an atomic representing
measure $\Phi=\sum_{i=1}^r W_i\delta_{\bx^{(i)}}$ with
$W_i\in\bS_+^{m}$, $\bx^{(i)}\in\RR^n$ and $\sum_{i=1}^r\rank(W_i)=\rank(M_k(\mS))$.
We need to prove $\bx^{(i)}\in\X$ for $i=1,\ldots,r$.
By Theorem \ref{th::FE}, we can extend $\tilde{\mS}$ to an infinite sequence 
$\hat{\mS}=(\hat{S}_{\ba})_{\ba\in\N^n}$ such that 
$M(\hat{\mS})\succeq 0$ and $\rank(M(\hat{\mS}))=\rank(M_k(\mS))$. For simplicity, in the following we will still use the symbol $\mS$ to denote $\hat{\mS}$.

For a column vector of polynomials $H(\bx)\in\RR[\bx]^m$, we write 
$H(\bx)=\sum_{\bg\in\supp{H}} H_{\bg}\bx^{\bg}$ with $H_{\bg}\in\RR^m$.
We define a subspace $\mathcal{I}_{\mS}$ of $\RR[\bx]^m$ associated with $\mS$ by
\[\mathcal{I}_{\mS}\coloneqq\left\{H(\bx)\in\RR[\bx]^m\ \middle\vert \ \sum_{\bg\in\supp{H}}S_{\ba+\bg}H_{\bg}=0,\ \forall\ba\in\N^n\right\}.\]
For any $H(\bx)\in\mathcal{I}_{\mS}$, we have $H(\bx^{(i)})^{\intercal}W_iH(\bx^{(i)})=0$ 
for all $i=1,\ldots,r$. In fact, as $H(\bx)\in\mathcal{I}_{\mS}$, it holds that
\[
0=\sum_{\ba\in \supp{H}}\sum_{\bb\in\supp{H}}
H_{\ba}^{\intercal}S_{\ba+\bb}H_{\bb}=\sum_{i=1}^rH(\bx^{(i)})^{\intercal}W_iH(\bx^{(i)}).
\]
As $W_i$'s are PSD, 
it implies that 
\begin{equation}\label{eq::QH}
H(\bx^{(i)})^{\intercal}W_iH(\bx^{(i)})=0\quad\text{and}\quad W_iH(\bx^{(i)})=0.
\end{equation}

Consider the quotient space $\RR[\bx]^m/\mathcal{I}_{\mS}\coloneqq\{H+\mathcal{I}_{\mS} \mid H\in\RR[\bx]^m\}$ over $\RR$
consisting of equivalence classes modulo $\mathcal{I}_{\mS}$. Let $t\coloneqq\rank(M(\mS))=\rank(M_k(\mS))$. 
Let $\bb^{(1)},\ldots,\bb^{(t)}\in\N^n_k$ (not necessarily distinct) and standard basis (column) vectors $\be^{(1)},\ldots, \be^{(t)}$ (not necessarily distinct) of $\RR^m$ be such that
\begin{equation}\label{eq::liset}
\left\{\col\left((S_{\ba+\bb^{(1)}})_{\ba\in\N^n}\right)
\be^{(1)}, \ldots, \col\left((S_{\ba+\bb^{(t)}})_{\ba\in\N^n}\right)
\be^{(t)} \right\}
\end{equation}
is a set of $t$ linearly independent column vectors of $M(\mS)$ and hence forms a basis of the column space of $M(\mS)$. 
Here, $\col\left((S_{\ba+\bb^{(i)}})_{\ba\in\N^n}\right)$ denotes the block-column vector with 
infinite $m\times m$ block entries $(S_{\ba+\bb^{(i)}})_{\ba\in\N^n}$. We claim that the set
\begin{equation}\label{eq::basis}
\left\{\bx^{\bb^{(1)}}\be^{(1)}+\mathcal{I}_{\mS},\ \ldots,\ 
\bx^{\bb^{(t)}}\be^{(t)}+\mathcal{I}_{\mS}\right\}
\end{equation}
forms a basis of $\RR[\bx]^m/\mathcal{I}_{\mS}$. To see this, first note that the elements in \eqref{eq::basis} are linearly independent as the elements in \eqref{eq::liset} are linearly independent. Then, it is sufficient to prove that for arbitrary $\bg\in\N^n$ and $j\in\N$ with $1\le j\le t$, the element $\bx^{\bg}\be^{(j)}+\mathcal{I}_{\mS}$ can be written as a linear combination of elements in \eqref{eq::basis}. This is indeed true since the column vector $\col\left((S_{\ba+\bg)_{\ba\in\N^n}}\right)\be^{(j)}$ can be written as a linear combination of elements in \eqref{eq::liset}. For any $H(\bx)\in\RR[\bx]^m$, we write $H(\bx)=H^{(0)}(\bx)+H^{(1)}(\bx)$ where $H^{(0)}$ is the residue of $H$ modulo $\mathcal{I}_{\mS}$ w.r.t. the basis \eqref{eq::basis} and $H^{(1)}\in \mathcal{I}_{\mS}$. Since $\bb^{(1)},\ldots,\bb^{(t)}\in\N^n_k$, we see that $\deg(H^{(0)})\le k$. 

Let $\{p^{(i)}(\bx)\}_{i=1}^r$ be the Lagrange interpolation polynomials at the points $\{\bx^{(i)}\}_{i=1}^r$ such that $p^{(i)}(\bx^{(i)})=1$ and $p^{(i)}(\bx^{(j)})=0$ for all $j\neq i$. Now we fix an $i$ and prove $\bx^{(i)}\in\X$. As $W_i\succeq 0$, there exists a vector $\bv^{(i)}\in\RR^m$ such that $(\bv^{(i)})^\intercal W_i\bv^{(i)}>0$. Let $H_i(\bx)=p^{(i)}(\bx) \bv^{(i)}\in\RR[\bx]^m$. 
Let us write $H_i=H^{(0)}_i+H_i^{(1)}$ and $H^{(0)}_i(\bx)=\sum_{\ba\in\supp{H^{(0)}_i}}H^{(0)}_{i,\ba}\bx^{\ba}$ with $\supp{H^{(0)}_i}\subseteq\N^n_k$. As $M_k(G\mS)\succeq 0$, we have 
\[\sum_{\ba\in\supp{H_i^{(0)}}}\sum_{\bb\in\supp{H_i^{(0)}}}
\left((H_{i,\ba}^{(0)})^{\intercal}\otimes I_q\right) [M_k(G\mS)]_{\ba\bb}
\left(H^{(0)}_{i,\bb}\otimes I_q\right)
\succeq 0.\]
By the definition of $M_k(G\mS)$, we have
\[
\begin{aligned}
&\sum_{\ba\in\supp{H_i^{(0)}}}\sum_{\bb\in\supp{H_i^{(0)}}}
\left((H_{i,\ba}^{(0)})^{\intercal}\otimes I_q\right) [M_k(G\mS)]_{\ba\bb}
\left(H^{(0)}_{i,\bb}\otimes I_q\right)\\
=&\sum_{\ba\in\supp{H_i^{(0)}}}\sum_{\bb\in\supp{H_i^{(0)}}}
\left((H_{i,\ba}^{(0)})^{\intercal}\otimes I_q\right) \left(
\sum_{\bg\in\supp{G}}S_{\ba+\bb+\bg}\otimes G_{\bg}\right) \left(H^{(0)}_{i,\bb}\otimes I_q\right)\\
=&\sum_{\bg\in\supp{G}}\left(\sum_{\ba\in\supp{H_i^{(0)}}}
\sum_{\bb\in\supp{H_i^{(0)}}}(H^{(0)}_{i,\ba})^{\intercal}
S_{\ba+\bb+\bg}H^{(0)}_{i,\bb} \right) G_{\bg}\\
=&\sum_{j=1}^r \left(H_i^{(0)}(\bx^{(j)})^{\intercal}W_jH_i^{(0)}(\bx^{(j)})\right)
G(\bx^{(j)})\\
=&\sum_{j=1}^r \left(\left(H_i(\bx^{(j)})-H_i^{(1)}(\bx^{(j)})\right)^{\intercal}W_j
\left(H_i(\bx^{(j)})-H_i^{(1)}(\bx^{(j)})\right)\right)
G(\bx^{(j)})\\
=&\sum_{j=1}^r \left(\left(p^{(i)}(\bx^{(j)}) \bv^{(i)}-H_i^{(1)}(\bx^{(j)})\right)^{\intercal}W_j
\left(p^{(i)}(\bx^{(j)}) \bv^{(i)}-H_i^{(1)}(\bx^{(j)})\right)\right)
G(\bx^{(j)})\\
=&\left(\bv^{(i)}-H_i^{(1)}(\bx^{(i)})\right)^{\intercal}W_i\left(\bv^{(i)}-H_i^{(1)}(\bx^{(i)})\right)G(\bx^{(i)})+\sum_{j\neq i}\left(H_i^{(1)}(\bx^{(j)})\right)^{\intercal}W_jH_i^{(1)}(\bx^{(j)})G(\bx^{(j)})\\
=&\left((\bv^{(i)})^\intercal W_i{\bv^{(i)}}\right)G(\bx^{(i)}),
\end{aligned}
\]
where the second-to-last equality is due to the fact that 
$p^{(i)}(\bx^{(i)})=1$ and $p^{(i)}(\bx^{(j)})=0$ for all $j\neq i$,
and the last equality is due to \eqref{eq::QH} and $H^{(1)}_i\in\mathcal{I}$. 
As $(\bv^{(i)})^\intercal W_i\bv^{(i)}>0$, we have $G(\bx^{(i)})\succeq 0$, which implies $\bx^{(i)}\in\X$ as desired.
\end{proof}
\begin{remark}\label{rk::FEC}
The \emph{block rank} $\rank_{bl}(\mS)$ of $M_k(\mS)$ is defined as the maximal number of linearly independent block columns (with $M_k(\mS)$ being regarded as a block matrix of entries $S_{\ba+\bb}$). If the FEC $\rank(M_k(\mS))=\rank(M_{k+d_G}(\tilde{\mS}))$ in Theorem \ref{th::FEC} is replaced by $\rank_{bl}(M_k(\mS))=\rank_{bl}(M_{k+d_G}(\tilde{\mS}))$, 
Nie \cite[Theorem 10.3.5]{NieBook} proved that $\mS$ admits an atomic representing measure.
We point out that the FEC on the block rank is stronger than that on the usual rank, and so our Theorem \ref{th::FEC} is more general.
In fact, it is easy to see that $\rank(M_k(\mS))=\rank(M_{k+d_G}(\tilde{\mS}))$ holds if 
$\rank_{bl}(M_k(\mS))=\rank_{bl}(M_{k+d_G}(\tilde{\mS}))$. However, the converse is not necessarily true. For example, letting $m=2$, $n=2$, $k=0$ and $d_G=1$, consider the sequence $\mS=(I_2)$ and its extension 
$\tilde{\mS}=(I_2, J, J, I_2, I_2, I_2)$ where
\[
I_2=\left[\begin{array}{cc}
1 &0\\
0 & 1
\end{array}\right]\quad\text{and}\quad
J=\left[\begin{array}{cc}
0 &1\\
1 & 0
\end{array}\right].
\]
We have $\rank(M_0(\mS))=\rank(M_1(\tilde{\mS}))=2$, but
$1=\rank_{bl}(M_0(\mS))\neq \rank_{bl}(M_1(\tilde{\mS}))=2$.
\end{remark}

\subsection{Matrix-valued measure recovery}\label{sec::recover}
For a truncated sequence $\mS=(S_{\ba})_{\ba\in\N^n_{2k}}\subseteq\bS^{m}$
with $k\ge d_G$, suppose that $M_k(\mS)\succeq 0$, $M_{k-d_G}(G\mS)\succeq 0$ and 
$\rank(M_k(\mS))=\rank(M_{k-d_G}(\mS))$. By Theorem \ref{th::FEC}, 
$\mS$ admits a finitely atomic representing measure $\Phi=\sum_{i=1}^r W_i\delta_{\bx^{(i)}}$
with $W_i\in\bS_+^{m}$, $\bx^{(i)}\in\X$ 
and $\sum_{i=1}^r\rank(W_i)=\rank(M_k(\mS))$. 
In theory, it was shown in \cite{KT2022} that the points $\{\bx^{(i)}\}_i$ can be computed via the intersecting zeros of the determinants of matrix-valued polynomials describing the flat extension.
In this subsection, inspired by \cite{LasserreHenrion}, we provide a linear algebra procedure for extracting $\bx^{(i)}\in\X$ and $W_i\in\bS^{m}$, which is much cheaper to implement. 

From the definition of $M_k(\mS)$, it holds
\[
M_k(\mS)=\sum_{i=1}^r\left(u_k(\bx^{(i)})\otimes I_m\right)W_i\left(u_k(\bx^{(i)})\otimes I_m\right)^\intercal,
\]
where $u_k(\bx^{(i)})$ is defined in \eqref{eq::ud}.
Letting $m_i=\rank(W_i)$, as $W_i$ is PSD, we have the decomposition 
$W_i=\sum_{j=1}^{m_i}\bw^{(i,j)}(\bw^{(i,j)})^{\intercal}$
for some $\bw^{(i,j)}\in\RR^m$. Then, we can write $M_k(\mS)=VV^\intercal$ with
\[
V=\left[\left(u_k(\bx^{(1)})\otimes I_m\right)[\bw^{(1,1)},\ldots,\bw^{(1,m_1)}],\ldots,
\left(u_k(\bx^{(r)})\otimes 
I_m\right)[\bw^{(r,1)},\ldots,\bw^{(r,m_r)}]\right].
\]

Let $M_k(\mS)=\widetilde{V}\widetilde{V}^\intercal$ be a Cholesky decomposition of $M_k(\mS)$ with $\widetilde{V}\in\RR^{m\snd{k}\times t}$ and $t=\rank(M_k(\mS))$. Notice that $V$ and $\widetilde{V}$ 
span the same column space. 
We will recover $\bx^{(i)}$ by suitable column operations on $\widetilde{V}$. 

Note that each column of $V$ is of form 
$u_k(\bx^{(i)})\otimes \bw^{(i,j)}$ and can be generated by
the columns of $\widetilde{V}$.
Now we treat the entries in the vectors $\bw^{(i,j)}$ as variables 
and denote it by  $\bw=(w_1,\ldots,w_m)$. Then,
the rows in $V$ correspond to the monomials
\[
v_k(\bx, \bw)=[\bw,\ x_1\bw,\ x_2\bw,\ \cdots,\ x_n\bw,\ x_1^2\bw,\ x_1x_2\bw,\ \cdots, \
x_n^k\bw]^{\intercal}.
\]

Reduce the matrix $\widetilde{V}$ to the column echelon form $U$:
\[
U=\left[
\begin{array}{ccccc}
   1   &   &   &   &  \\
   \star &   &&&\\
   0 & 1 & &&\\
   0 & 0 & 1& &\\
   \star & \star&\star & &\\
   &\vdots & & \ddots & \\
   0 & 0 & 0 &\cdots & 1\\
   \star & \star&\star &\cdots  &\star\\
   & \vdots &&& \vdots\\
   \star & \star&\star &\cdots  &\star\\
\end{array}
\right].
\]
From the rows of $U$ where the pivot elements locate, we obtain a (column) 
monomial basis $b_k(\bx,\bw)$ which consists of $t$ monomials in 
$v_k(\bx, \bw)$ such that
\begin{equation}\label{eq::vUb}
v_k(\bx, \bw)=U b_k(\bx, \bw)
\end{equation}
holds at each pair $(\bx^{(i)}, \bw^{(i,j)})$, 
$j=1,\ldots,m_i$, $i=1,\ldots,r$. Note that each monomial $\bx^{\ba}w_j$
in $b_k(\bx, \bw)$ satisfies $\vert\ba\vert\le k-d_G$ since
$\rank(M_k(\mS))=\rank(M_{k-d_G}(\mS))$.
\begin{proposition}\label{prop::LD}
   The vectors $b_k(\bx^{(i)}, \bw^{(i,j)})$, $j=1,\ldots,m_i$, $i=1,\ldots,r$, 
   are linearly independent.
\end{proposition}
\begin{proof}
Case 1: $r-1\le k$.
Suppose on the contrary that $b_k(\bx^{(i)}, \bw^{(i,j)})$, $j=1,\ldots,m_i$, $i=1,\ldots,r$ are linearly dependent. Then there exist constants $c_{i,j}$'s, not all zeros, such that
\[\sum_{i=1}^r\sum_{j=1}^{m_i} c_{i,j} b_k(\bx^{(i)}, \bw^{(i,j)})=0.\]
Because the points $\bx^{(i)}$'s are distinct, we can construct the Lagrange interpolation polynomials $p^{(i)}(\bx)$'s at $\bx^{(i)}$'s such that $p^{(i)}(\bx^{(i)})=1$ and 
$p^{(i)}(\bx^{(j)})=0$ for all $i\neq j$. 
Now we fix an $i'$ with $1\le i'\le r$, and consider the column vector of polynomials $\bw p^{(i')}(\bx)\in\RR[\bx, \bw]^m$.
As $\deg(p^{(i)})=r-1\le k$, due to \eqref{eq::vUb}, there exists a coefficient matrix
$\Xi\in\RR^{m\times t}$ such that $\bw p^{(i')}(\bx)=\Xi b_k(\bx,\bw)$
holds at each pair $(\bx^{(i)}, \bw^{(i,j)})$, $j=1,\ldots,m_i$, $i=1,\ldots,r$. Then, we have
\[
\begin{aligned}
0=\Xi\left(\sum_{i=1}^r\sum_{j=1}^{m_i} c_{i,j} b_k(\bx^{(i)}, \bw^{(i,j)})\right)
=\sum_{i=1}^r\sum_{j=1}^{m_i} c_{i,j} \bw^{(i,j)}p^{(i')}(\bx^{(i)}) 
=\sum_{j=1}^{m_{i'}}c_{i',j}\bw^{(i',j)}.
\end{aligned}
\]
As $\bw^{(i',j)}$'s are linearly independent, we have $c_{i',j}=0$ for all $j=1,\ldots,m_i$.
This leads to a contradiction, since $i'$ can be arbitrarily chosen.

Case 2: $r-1 > k$. According to Theorem \ref{th::FE},
$\mS$ admits a flat extension $\tilde{\mS}=(\tilde{S}_{\ba})_{\ba\in\N^n_{2r-2}}$
such that $M_{r-1}(\tilde{\mS})\succeq 0$ and $\rank(M_k(\mS))=\rank(M_{r-1}(\tilde{\mS}))$.
By repeating the previous arguments on $M_{r-1}(\tilde{\mS})$, we can still obtain
a column echelon form $\tilde{U}$ and a monomial basis $b_{r-1}(\bx,\bw)$ such that
$v_{r-1}(\bx, \bw)=\tilde{U} b_{r-1}(\bx, \bw)$ holds at all pairs $(\bx^{(i)}, \bw^{(i,j)})$'s. 
Since $M_{r-1}(\tilde{\mS})$ is a flat extension of $M_k(\mS)$, it is easy to see that
the basis $b_{r-1}(\bx,\bw)$ is identical to $b_k(\bx,\bw)$.
Now as in Case 1, we can show that $b_{r-1}(\bx^{(i)}, \bw^{(i,j)})$ 
and hence $b_k(\bx^{(i)}, \bw^{(i,j)})$, $j=1,\ldots,m_i$, $i=1,\ldots,r$,
are linearly independent. 
\end{proof}

Recall that each monomial $\bx^{\ba}w_j$ in $b_k(\bx, \bw)$ satisfies
$\vert\ba\vert\le k-d_G<k$. Hence, for each $l=1, \ldots, n$, we can
extract from $U$ the $t\times t$ multiplication matrix $N_l$ such 
that $N_l b_k(\bx, \bw)=x_l b_k(\bx, \bw)$ holds at each pair $(\bx^{(i)}, \bw^{(i,j)})$, 
$j=1,\ldots,m_i$, $i=1,\ldots,r$.

Following \cite{CGT1997}, we build a random combination of multiplication matrices
$N=\sum_{l=1}^n c_lN_l$, where $c_l>0$ and $\sum_{l=1}^n c_l=1$. Let $N=ATA^\intercal$ be the ordered Schur decomposition
of $N$, where $A=[a_1,\ldots,a_t]$ is an orthogonal matrix with $A^\intercal A=I_t$ and $T$ is upper-triangular with eigenvalues of $N$ being sorted increasingly along the diagonal.
\begin{proposition}
Suppose that the constants $c_l$'s are chosen such that 
$h(\bx)=\sum_{l=1}^n c_lx_l$ takes distinct values on $\bx^{(i)}$, $i=1,\ldots,r$. Then the set of points
\begin{equation}\label{eq::lpoints}
\bigl\{(a_1^\intercal N_1 a_1, \ldots, a_1^\intercal N_n a_1),\ \ldots,\
(a_t^\intercal N_1 a_t, \ldots, a_t^\intercal N_n a_t)\bigr\}
\end{equation}
is exactly $\{\bx^{(1)},\ldots,\bx^{(r)}\}$ and each $\bx^{(i)}$ appears
$m_i=\rank(W_i)$ times.
\end{proposition}
\begin{proof}
Since $N_l b_k(\bx, \bw)=x_l b_k(\bx, \bw)$ holds at each pair
$(\bx^{(i)}, \bw^{(i,j)})$, $j=1,\ldots,m_i$, $i=1,\ldots,r$.
It is easy to see that
\[h(\bx^{(i)})b_k(\bx^{(i)}, \bw^{(i,j)})
=\sum_{l=1}^n c_l x^{(i)}_l b_k(\bx^{(i)}, \bw^{(i,j)})
=N b_k(\bx^{(i)}, \bw^{(i,j)}),
\]
for each $j=1,\ldots,m_i$, $i=1,\ldots,r$. In other words, 
for each $i=1,\ldots,r$, $h(\bx^{(i)})$ is an eigenvalue of $N$ and $b_k(\bx^{(i)}, \bw^{(i,j)})$, $j=1,\ldots,m_i$ are the associated eigenvectors.
By Proposition \ref{prop::LD}, the $t$ vectors $b_k(\bx^{(i)}, \bw^{(i,j)})$, $j=1,\ldots,m_i$, $i=1,\ldots,r$ are linearly independent. Therefore, $\{h(\bx^{(1)}),\ldots,h(\bx^{(r)})\}$ is exactly the set of eigenvalues of $N$, and $\{b_k(\bx^{(i)}, \bw^{(i,j)}), j=1,\ldots,m_i\}$ 
spans the eigenspace of $N$ associated with $h(\bx^{(i)})$. 
So, we can divide the set $\{a_1,\ldots,a_t\}$ into $r$ groups $\mathcal{A}_1,\ldots,\mathcal{A}_r$ with $\lvert\mathcal{A}_i\rvert=m_i$, such that $\mathcal{A}_i$ spans the eigenspace of $N$ associated with $h(\bx^{(i)})$. Now fix an $i$ and 
a vector $a\in\mathcal{A}_i$. There exist weights $\lambda_1,\ldots,\lambda_{m_i}\in\RR$ such that 
$a=\sum_{j=1}^{m_i}\lambda_{j} b_k(\bx^{(i)}, \bw^{(i,j)}).$
Then, for each $l=1,\ldots,n$, it holds
\[
\begin{aligned}
a^\intercal N_l a&= 
\left(\sum_{j=1}^{m_i}\lambda_{j} b_k(\bx^{(i)}, \bw^{(i,j)})\right)^\intercal
N_l
\left(\sum_{j=1}^{m_i}\lambda_{j} b_k(\bx^{(i)}, \bw^{(i,j)})\right)\\
&=\left(\sum_{j=1}^{m_i}\lambda_{j} b_k(\bx^{(i)}, \bw^{(i,j)})\right)^\intercal
\left(\sum_{j=1}^{m_i}\lambda_{j} x^{(i)}_l b_k(\bx^{(i)}, \bw^{(i,j)})\right)\\
&=x^{(i)}_l a^\intercal a=x^{(i)}_l.
\end{aligned}
\]
Hence, $(a^\intercal N_1 a, \ldots, a^\intercal N_n a) = \bx^{(i)}$.
The conclusion then follows.
\end{proof}

Once the points $\bx^{(1)},\ldots,\bx^{(r)}$ are obtained, let
\[
\Lambda\coloneqq\left[u_k(\bx^{(1)}), \ldots, u_k(\bx^{(r)})\right]\otimes I_m\in\RR^{m\snd{k}\times mr},
\]
and we have  
\begin{equation}\label{eq::MS}
M_k(\mS)=\Lambda\diag(W_1,\ldots,W_r)\Lambda^\intercal.
\end{equation}
Notice that the first $m$ columns of $\diag(W_1,\ldots,W_r) \Lambda^\intercal$ are exactly $[W_1,\ldots,W_r]^\intercal$. By comparing the first $m$ columns of both sides of \eqref{eq::MS}, we get
\begin{equation}\label{eq::MS2}
\col(\{S_{\ba}\}_{\ba\in\N^n_k})=\Lambda [W_1,\ldots,W_r]^\intercal.
\end{equation}
Assume that $\Lambda$ has $mr$ independent rows (see Remark \ref{rk::indepen}) and let $\mathcal{R}$ be the index set of these rows. 
Denote by $\Lambda_{\mathcal{R}}$ (resp. $M_{\mathcal{R}}(\mS)$) the $mr\times mr$
(resp. $mr\times m$) submatrix of $\Lambda$ 
(resp. $\col(\{S_{\ba}\}_{\ba\in\N^n_k})$) 
whose rows are indexed by $\mathcal{R}$. Then by extracting the rows indexed $\mathcal{R}$ from both sides of \eqref{eq::MS2}, we have
$M_{\mathcal{R}}(\mS)=\Lambda_{\mathcal{R}} [W_1,\ldots,W_r]^\intercal.$
Hence, the matrices $W_i$'s can be retrieved by
$[W_1,\ldots,W_r]^\intercal=\Lambda_{\mathcal{R}}^{-1}  M_{\mathcal{R}}(\mS).$
We provide an example illustrating the above procedure in Appendix \ref{secA1}.
\begin{remark}\label{rk::indepen}
It is clear that if 
\begin{equation}\label{eq::rkuk}
\rank([u_k(\bx^{(1)}), \ldots, u_k(\bx^{(r)})])=r,
\end{equation}
then $\Lambda$ must have $mr$ independent rows. As $\bx^{(i)}$'s are distinct,
by considering the Lagrange interpolation polynomials at $\bx^{(i)}$'s, 
we know that \eqref{eq::rkuk} always holds if $k\ge r-1$. If $k< r-1$, then it is possible that \eqref{eq::rkuk} fails
and we may need to consider flat extensions of $M_k(\mS)$
to recover the weights $W_i$'s. 
Inspired by \cite{Nie2014}, 
we propose the following heuristic method for finding such a flat extension. Consider the linear matrix-valued moment optimization problem:
\[
\min_{\widetilde{\mS}}\ \sum_{\ba\in\N^n_{2d}}\tr{R_{\ba}\widetilde{S}_{\ba}}\  \text{  s.t. }\
\widetilde{\mS}=(\widetilde{S}_{\ba})_{\ba\in\N^n_{2d}}\in\ms_d^m(G),\ \widetilde{S}_{\ba}=S_{\ba},\ \forall 
\ba\in\N^n_{2k},
\]
for some $d>k$ and $(R_{\ba})_{\ba\in\N^n_{2d}}\subseteq \bS^{m}$. By randomly choosing the matrices $R_{\ba}$, one may expect that the optimum of the above problem is achieved at an extreme point of the cone $\ms_d^m(G)$ which might admit a finitely atomic matrix-valued representing measure and hence provide a flat extension of $M_k(\mS)$. We leave a detailed study on this issue in the future.

\end{remark}

\vskip 8pt
To conclude this section, we would like to remark that the results on the truncated matrix-valued $\X$-moment problem and the procedure for matrix-valued measure recovery are not direct generalizations of the scalar case. First, the matrix-valued moment problem cannot be reduced to the scalar case via certain scalarizing procedures as far as we know. For instance, it is not valid to solve the matrix-valued moment problem by separately considering the sequence of the entries at the same position in the matrices $\{S_{\ba}\}_{\ba\in\N^n_{2k}}$. This is because the off-diagonal entries of a PSD matrix-valued measure are not necessarily positive scalar measures, making the theory of the scalar moment problem inapplicable here. One may also consider the scalar sequence $\{\tr{S_{\ba}}\}_{\ba\in\N^n_{2k}}$ which, however, contains only partial information on the matrix-valued moments. It is thus impossible to recover the matrix-valued measure in this way. 
Second, extending the theory of the scalar moment problem to the matrix-valued case
requires carefully addressing additional complexities and intrinsic challenges present in the matrix-valued setting. 
In particular, the proof of Theorem \ref{th::FEC} is conducted in the quotient space $\RR[\bx]^m/\mathcal{I}_{\mS}$ due to the matrix-valued nature, which involves 
a subtle construction of the polynomial vectors $H_i$. In addition, the linear algebra procedure presented in Section \ref{sec::recover} is the first efficient approach for extracting matrix-valued measures from the matrix-valued moment matrix satisfying FEC. 
In comparison with the scalar counterpart in \cite{LasserreHenrion}, this procedure demands several delicate techniques tailored to the matrix-valued setting to handle the variables introduced by both the points $\bx^{(i)}$ and matrices $W_i$, making the justification of its correctness much more involved.

\section{A moment-SOS hierarchy for \eqref{eq::RCPSO} with SOS-convexity}\label{sec::CRSDP}
In this section, we first reformulate \eqref{eq::RCPSO} as a conic optimization problem, based on which we can then derive a moment-SOS hierarchy whose optima monotonically converge to the optimum of \eqref{eq::RCPSO}. Furthermore, the results in Section \ref{sec::mrecovery} enable us to detect finite convergence of the moment-SOS hierarchy and to extract optimal solutions.
%

\subsection{A conic reformulation}
For simplicity, we write
\[
P(\by, \bx)=\sum_{\ba\in\N^n} P_{\ba}(\by)\bx^{\ba}=\sum_{\bb\in\N^{\ell}} P_{\bb}(\bx) \by^{\bb},
\]
where $P_{\ba}(\by)\in\bS[\by]^{m}$ (resp. $P_{\bb}(\bx)\in\bS[\bx]^{m}$) is the coefficient 
matrix of $\bx^{\ba}$ (resp. $\by^{\bb}$) with $P(\bx,\by)$ being regarded as a polynomial matrix in $\bS[\bx]^{m}$ (resp. $\bS[\by]^{m}$). 
For a linear functional $\mL \colon \bS[\bx]^{m} \to \RR$, we let
\[
\mL(P(\by, \bx))\coloneqq\sum_{\bb\in\N^{\ell}} \mL(P_{\bb}(\bx)) \by^{\bb}\in\RR[\by],
\]
and for a linear functional $\mH \colon \RR[\by] \to \RR$, we let
\[
\mH(P(\by, \bx))\coloneqq\sum_{\bb\in\N^{\ell}} P_{\bb}(\bx) \mH(\by^{\bb})\in\bS[\bx]^{m}.
\]

To obtain a conic reformulation of \eqref{eq::RCPSO}, we need to assume the Slater condition to hold.
\begin{assumption}\label{assump3}
{\rm
The Slater condition holds for \eqref{eq::RCPSO}, 
i.e., there exists $\bar{\by}\in\Y$ such that $\theta_i(\bar{\by})>0$ for all $i=1,\ldots,s$, 
and $P(\bar{\by}, \bx)\succ 0$ for all $\bx\in\X$.}
\end{assumption}

\begin{proposition}\label{prop::sduality}
Under Assumptions \ref{assump1} and \ref{assump3},
there exists a finitely atomic matrix-valued measure $\Phi^{\star}\in\mathfrak{M}^m_+(\X)$ such that 
\[
f^{\star}=\inf_{\by\in\Y}\ f(\by)-\mL_{\Phi^{\star}}(P(\by, \bx)).
\]
If $\by^{\star}$ is an optimal solution to \eqref{eq::RCPSO}, then 
$\mL_{\Phi^{\star}}(P(\by^{\star}, \bx))=0$.
\end{proposition}
\begin{proof}
For any $\bv\in \V\coloneqq\{\bv\in\RR^m \mid \sum_{i=1}^m v_i^2=1\}$ and $\bx\in\X$, by Assumption \ref{assump1}(ii), we have that the function $-\bv^{\intercal}P(\,\cdot\,, \bx)\bv$ is convex in $\by$. Then, \eqref{eq::RCPSO} can be equivalently reformulated as the following convex semi-infinite program under Assumption \ref{assump1}:
\begin{equation}\label{eq::SIP}
f^{\star}=\inf_{\by\in\Y}\ f(\by)\quad \text{s.t.}\ \bv^{\intercal}P(\by, \bx)\bv\ge 0,\ \forall (\bx, \bv)\in \X\times\V. 
\end{equation}
Let $(\bx^{(0)}, \bv^{(0)}), (\bx^{(1)}, \bv^{(1)}), \ldots, (\bx^{(\ell)}, \bv^{(\ell)})$ be $\ell+1$ 
arbitrary points in $\X\times\V$. By Assumption \ref{assump3}, there exists $\bar{\by}\in\Y$  such that 
$P(\bar{\by}, \bx^{(i)})\succ 0$ for all $i=0, 1, \ldots, \ell$. Hence, it holds
$(\bv^{(i)})^{\intercal}P(\bar{\by}, \bx^{(i)})\bv^{(i)}>0$ for all $i=0, 1, \ldots, \ell$. As
$\X\times\V$ is compact in $\RR^n\times\RR^m$ (Assumption \ref{assump1}(iii)), 
recall by \cite[Theorem 4.1]{Borwein1981} that if for any arbitrary $\ell+1$
points in $\X\times\V$, the corresponding constraints in \eqref{eq::SIP} are strictly satisfied 
by some point from $\Y$, then \eqref{eq::SIP} is reducible. In other words, there
exist $\ell$ points $(\tilde{\bx}^{(1)}, \tilde{\bv}^{(1)}), 
\ldots, (\tilde{\bx}^{(\ell)}, \tilde{\bv}^{(\ell)})\in\X\times \V$ such that 
\[
f^{\star}=\inf_{\by\in\Y}\ f(\by)\quad \text{s.t.}\ 
(\tilde{\bv}^{(i)})^{\intercal}P(\by, \tilde{\bx}^{(i)})\tilde{\bv}^{(i)}\ge 0, \ i=1,
\ldots, \ell.
\]
Then, by the Lagrange multiplier theorem, 
we can find 
$\lambda_1,\ldots,\lambda_{\ell}>0$ such that
\[
f^{\star}=\inf_{\by\in\Y}\ f(\by)-\sum_{i=1}^{\ell}\lambda_i 
(\tilde{\bv}^{(i)})^{\intercal}P(\by, \tilde{\bx}^{(i)})\tilde{\bv}^{(i)}.
\]
Define a finitely atomic matrix-valued measure 
\[
\Phi^{\star}\coloneqq \sum_{i=1}^{\ell}\lambda_i\tilde{\bv}^{(i)}(\tilde{\bv}^{(i)})^{\intercal} \delta_{\tilde{\bx}^{(i)}}\in \mathfrak{M}^m_+(\X).
\]
Then, it holds that
\[
f^{\star}=\inf_{\by\in\Y}\ f(\by)-\mL_{\Phi^{\star}}(P(\by, \bx)).
\]
If $\by^{\star}$ is an optimal solution to \eqref{eq::RCPSO}, then 
$P(\by^{\star}, \bx)\succeq 0$ on $\X$ and hence 
\[
f^{\star}\le f(\by^{\star})-\mL_{\Phi^{\star}}(P(\by^{\star}, \bx))
\le f(\by^{\star})=f^{\star},
\]
which implies $\mL_{\Phi^{\star}}(P(\by^{\star}, \bx))=0$. 
\end{proof}

\begin{remark}
From the above proof, one can see that Proposition \ref{prop::sduality} remains true if the Slater condition is weakened as: For any $\ell+1$ points $\bx^{(0)}, \bx^{(1)}, \ldots, \bx^{(\ell)}\in\X$, there exists $\bar{\by}\in\Y$ such that $P(\bar{\by}, \bx^{(i)})\succ 0$ for all $i=0, 1, \ldots, \ell$.
\end{remark}


For any measure $\mu\in\mathfrak{m}_+(\Y)$, we define an associated linear functional $\mH_{\mu} \colon \RR[\by] \to \RR$ by $\mH_{\mu}(h)=\int_{\Y} h(\by)\ud \mu(\by)$ for all $h\in\RR[\by]$. Let us consider the following conic optimization problem:
\begin{equation}\label{eq::cp}
\left\{
   \begin{aligned}
     \tilde{f}\coloneqq\sup_{\rho, \Phi}&\  \rho\\
      \text{s.t.}&\ f(\by)-\rho-\mL_{\Phi}(P(\by, \bx))\in \mathcal{P}(\Y),\\
      &\ \rho\in\RR,\ \Phi\in\mathfrak{M}^m_+(\X),
   \end{aligned} 
   \right.
\end{equation}
whose dual reads as
\begin{equation}\label{eq::cd}
\left\{
   \begin{aligned}
     \hat{f}\coloneqq\inf_{\mu}\,&\  \mH_{\mu}(f)\\
      \text{s.t.}&\ \mu\in\mathfrak{m}_+(\Y),\ \mH_{\mu}(1)=1,\\
      &\ \mH_{\mu}(P(\by, \bx))\in\mathcal{P}^m(\X),
   \end{aligned} 
   \right.
\end{equation}
with $\mathcal{P}^m(\X)$ being defined in \eqref{eq::pm}.

\begin{theorem}\label{th::equivalent}
Under Assumptions \ref{assump1} and \ref{assump3}, it holds $\tilde{f}=\hat{f}=f^{\star}$.
\end{theorem}
\begin{proof}
Let $\Phi^{\star}\in\mathfrak{M}^m_+(\X)$ be the finitely atomic matrix-valued measure given in 
Proposition \ref{prop::sduality}. Then $(f^{\star}, \Phi^{\star})$ is feasible to \eqref{eq::cp} due to Proposition \ref{prop::sduality}. Thus, $\tilde{f}\ge f^{\star}$.
By the weak duality, we have $\hat{f}\ge\tilde{f}$. It remains to show $\hat{f}\le f^{\star}$.

Let $(\by^{(k)})_{k\in\N}$ be a minimizing sequence of \eqref{eq::RCPSO}.
Then, for any $\varepsilon>0$, there exists $k_{\varepsilon}\in\N$ such that $\by^{(k_{\varepsilon})}$ is feasible to \eqref{eq::RCPSO}
and $f(\by^{(k_{\varepsilon})})\le f^{\star}+\varepsilon$. 
The Dirac measure $\delta_{\by^{(k_{\varepsilon})}}$ centered at $\by^{(k_{\varepsilon})}$ is feasible to \eqref{eq::cd}.
Therefore, $\hat{f}\le \mH_{\delta_{\by^{(k_{\varepsilon})}}}(f)=f(\by^{(k_{\varepsilon})})  \le f^{\star}+\varepsilon.$ 
As $\varepsilon>0$ is arbitrary, we have $\hat{f}\le f^{\star}$ as desired. 
\end{proof}

\begin{remark}\label{rk::convex}
It is easy to verify that the conic reformulation \eqref{eq::cp}--\eqref{eq::cd} for 
\eqref{eq::RCPSO} and the results in Proposition \ref{prop::sduality} and 
Theorem \ref{th::equivalent} remain true 
if ``SOS-convex'' in Assumption \ref{assump1} (i) 
is weakened to ``convex'' and  ``PSD-SOS-convex'' in Assumption \ref{assump1} (ii) 
is weakened to ``PSD-convex''
(i.e., if Assumption \ref{assump1} is replaced by Assumption \ref{assump5} in Section \ref{sec::extension}).
\end{remark}
\begin{remark}
    Before we proceed, let us highlight the significance of the conic reformulation \eqref{eq::cp}--\eqref{eq::cd} 
    and explain how it could provide insights in solving \eqref{eq::RCPSO}.
    Essentially, through the reformulation \eqref{eq::cp}--\eqref{eq::cd}, the complexity of 
    \eqref{eq::RCPSO} is transferred to the conic constraints in \eqref{eq::cp} and  \eqref{eq::cd}, 
    which involve the cone $\mathcal{P}(\Y)$ of nonnegative polynomials on the set $\Y$,
    the cone $\mathfrak{M}^m_+(\X)$ of PSD matrix-valued measures supported on $\X$, and their dual cones. 
    While the exact representations of these cones are unavailable in general case, 
    there are various (matrix-valued) Positivstellens\"atze and the dual (matrix-valued) moment theories which can provide tractable approximations for these cones. Hence, by replacing these cones with suitable approximations, 
    the reformulation \eqref{eq::cp}--\eqref{eq::cd} offers us  
    a unified framework to derive tractable relaxations for \eqref{eq::RCPSO}. 
\end{remark}

\subsection{A moment-SOS hierarchy}
Let $\Theta\coloneqq\{\theta_1,\ldots, \theta_{s}\}\subseteq\RR[\by]$ collect the description polynomials of the semialgebraic set $\Y$.
Moreover, let
\[
\begin{aligned}
&k_{\by}\coloneqq\max\,\{\deg(f),\ \deg(\theta_1),\ldots, \deg(\theta_{s}), \ \deg_{\by}(P_{ij}),\ 
i,j=1,\ldots,m\},\\
&k_{\bx}\coloneqq\max\,\{\deg_{\bx}(P_{ij}),\ i,j=1,\ldots,m,\ \deg(G)\}.
\end{aligned}
\]
For each $k\ge\lceil k_{\bx}/2\rceil$, by replacing the cones $\mathcal{P}(\Y)$ and $\mathfrak{M}^m_+(\X)$ in \eqref{eq::cp} with the more tractable cones $\QM_{\lceil k_{\by}/2\rceil}(\Theta)$ and $\ms^m_k(G)$ respectively, we obtain the following SDP relaxation for \eqref{eq::RCPSO}:
\begin{equation}\label{eq::cpsdpSOS}
\left\{
   \begin{aligned}
     f^{\psdp}_k\coloneqq\sup_{\rho, \mS}&\  \rho\\
      \text{s.t.}&\ f(\by)-\rho-\mL_{\mS}(P(\by, \bx))\in 
      \QM_{\lceil k_{\by}/2\rceil}(\Theta),\\
      &\ \rho\in\RR,\ \mS\in\ms^m_k(G).
   \end{aligned} 
   \right.
\end{equation}
The dual of \eqref{eq::cpsdpSOS} reads as
\begin{equation}\label{eq::cdsdpSOS}
\left\{
   \begin{aligned}
     f^{\dsdp}_k\coloneqq\inf_{\bss}\,&\  \mH_{\bss}(f)\\
      \text{s.t.}&\ \bss\in\ms_{\lceil k_{\by}/2\rceil}(\Theta),\ \mH_{\bss}(1)=1,\\
      &\ \mH_{\bss}(P(\by, \bx))\in\QM^m_k(G),
   \end{aligned} 
   \right.
\end{equation}
where the linear functional $\mH_{\bss} \colon \RR[\by]_{2\lceil k_{\by}/2\rceil}\to \RR$ is defined by $\mH_{\bss}(h)=\sum_{\ba\in\supp{h}}h_{\ba}s_{\ba}$ for $h\in\RR[\by]_{2\lceil k_{\by}/2\rceil}$. 
We call \eqref{eq::cpsdpSOS}--\eqref{eq::cdsdpSOS} the \emph{moment-SOS hierarchy} for \eqref{eq::RCPSO} with SOS-convexity and call $k$ the \emph{relaxation order}. 

\subsubsection{Convergence analysis}
Next we present the convergence analysis of the moment-SOS hierarchy \eqref{eq::cpsdpSOS}--\eqref{eq::cdsdpSOS}. 
\begin{proposition}\label{prop::lagSOS}
Let Assumptions \ref{assump1} and \ref{assump3} hold.
Then the sequences $(f_k^{\psdp})_{k\ge\lceil k_{\bx}/2\rceil}$ and $(f_k^{\dsdp})_{k\ge\lceil k_{\bx}/2\rceil}$ are monotonically non-increasing and provide upper bounds on $f^{\star}$.
\end{proposition}
\begin{proof}
By the weak duality $f_k^{\psdp}\le f_k^{\dsdp}$, we only need to prove that there exists a finitely atomic matrix-valued measure 
$\Phi^{\star}\in\mathfrak{M}^m_+(\X)$ such that 
\begin{equation}\label{eq::LQ}
f(\by) - f^{\star} -\mL_{\Phi^{\star}}(P(\by, \bx))\in\QM_{\lceil k_{\by}/2\rceil}(\Theta).
\end{equation}
As Assumptions \ref{assump1} and \ref{assump3} hold, by Proposition \ref{prop::sduality}, 
there exists a finitely atomic matrix-valued measure $\Phi^{\star}\in\mathfrak{M}^m_+(\X)$ such that 

\begin{equation}\label{eq::scp}
f^{\star}=\inf_{\by\in\RR^{\ell}}\ f(\by)-\mL_{\Phi^{\star}}(P(\by, \bx)) \quad\text{s.t. } 
\theta_1(\by)\ge 0, \ldots, \theta_s(\by)\ge 0.
\end{equation}
Suppose that $\Phi^{\star}=\sum_{i=1}^r W_i\delta_{\bx^{(i)}}\in\mathfrak{M}^m_+(\X)$
for some $\bx^{(1)}, \ldots, \bx^{(r)}\in\X$ and
$W_1, \ldots, W_r\in\bS_+^{m}$. For each $i=1,\ldots, r$, let  
$W_i=\sum_{k=1}^{m_i} \bv^{(i,k)}(\bv^{(i,k)})^{\intercal}$ for some 
$\bv^{(i,k)}\in\RR^m$. Then, 
\[
\begin{aligned}
  f(\by) -\mL_{\Phi^{\star}}(P(\by, \bx))=
  f(\by)-\sum_{i=1}^r\tr{P(\by, \bx^{(i)})W_i}
  =f(\by)-\sum_{i=1}^r\sum_{k=1}^{m_i} 
  (\bv^{(i,k)})^{\intercal}P(\by, \bx^{(i)})\bv^{(i,k)}.
  \end{aligned}
\] 
By Assumption \ref{assump1} and the definition of PSD-SOS-convexity, 
$f(\by) -\mL_{\Phi^{\star}}(P(\by, \bx))$ is SOS-convex. 
By Proposition \ref{prop::sduality}, each optimal solution to \eqref{eq::RCPSO} is 
also optimal to \eqref{eq::scp}.
Since the functions $f(\by) -\mL_{\Phi^{\star}}(P(\by, \bx))$, $-\theta_i(\by)$'s
in \eqref{eq::scp} are all SOS-convex and the set of optimal solutions of \eqref{eq::scp}
is nonempty, \cite[Theorem 3.3]{LasserreConvex}
states that there exists a convex SOS polynomial $\sigma(\by)\in\RR[\by]$ and scalars
$\lambda_1\ge 0, \ldots, \lambda_s\ge 0$, such that
\[
f(\by) -f^{\star}-\mL_{\Phi^{\star}}(P(\by, \bx))=\sigma(\by)+\sum_{j=1}^s\lambda_j\theta_j(\by).
\]
Clearly, we have $\deg(\sigma)\le 2\lceil k_{\by}/2\rceil$ which implies that
\eqref{eq::LQ} holds true and the conclusion follows.

\end{proof}

We next prove the
asymptotic convergence of this moment-SOS hierarchy.
We write $\be=(\be_1,\ldots,\be_{\ell})$ for the standard basis of $\RR^{\ell}$ and let $\bss_{\be}=(s_{\be_1},\ldots,s_{\be_{\ell}})$ for any feasible 
point $\bss=(s_{\ba})_{\ba\in\N^{\ell}_{2\lceil k_{\by}/2\rceil}}$ of \eqref{eq::cdsdpSOS}.

\begin{proposition}\label{prop::js}
Suppose that $Q(\by)\in\bS[\by]^{m}$ is PSD-SOS-convex.
Let $\bss=(s_{\ba})_{\ba\in\N^{\ell}_{2\lceil\deg(Q)/2\rceil}}$
satisfy $s_{\mathbf{0}}=1$ and $M_{\lceil\deg(Q)/2\rceil}(\bss)\succeq 0$. Then
$\mH_{\bss}(Q) \succeq Q(\bss_{\be}).$
\end{proposition}
\begin{proof}
As $\bv^{\intercal}Q(\by)\bv$ is SOS-convex in $\by$ for all $\bv\in\RR^m$, by
\cite[Theorem 2.6]{LasserreConvex}, it holds
\[
\bv^{\intercal}\mH_{\bss}(Q(\by))\bv=
\mH_{\bss}(\bv^{\intercal}Q(\by)\bv) \ge \bv^{\intercal}Q(\bss_{\be})\bv,
\]
for all $\bv\in\RR^m$. Hence we have $\mH_{\bss}(Q) \succeq Q(\bss_{\be}).$
\end{proof}

\begin{corollary}\label{cor::feasible}
Suppose that $-\theta_1(\by),\ldots,-\theta_s(\by)$ are SOS-convex and 
$-P(\by,\bx)$ is PSD-SOS-convex in $\by$ for all $\bx\in\X$. If $\bss$ is feasible to \eqref{eq::cdsdpSOS}, then $\bss_{\be}\in\Y$ and it is feasible to \eqref{eq::RCPSO}.
\end{corollary}
\begin{proof}
By the extended Jensen’s inequality for SOS-convex polynomials 
\cite[Theorem 2.6]{LasserreConvex}, it holds 
$\theta_i(\bss_{\be})\ge \mH_{\bss}(\theta_i)\ge 0$
for $i=1,\ldots,s$, which implies $\bss_{\be}\in\Y$.
By Proposition \ref{prop::js}, for every $\bx\in\X$, we have 
$P(\bss_{\be}, \bx)\succeq \mH_{\bss}(P(\by, \bx))\succeq 0.$
So $\bss_{\be}$ is feasible to \eqref{eq::RCPSO}. 
\end{proof}

\begin{theorem}\label{th::asym_convergence_SOS}
Under Assumptions \ref{assump1}--\ref{assump3}, the following are true:
\begin{enumerate}
\item[\upshape (i)] $f^{\psdp}_k\searrow f^{\star}$
and $f^{\dsdp}_k\searrow f^{\star}$ as $k\rightarrow\infty$;
\item[\upshape (ii)] For any convergent subsequence $(\bss^{(k_i,\star)})_i$ 
of $(\bss^{(k,\star)})_k$ where each $\bss^{(k,\star)}$ is a minimizer of \eqref{eq::cdsdpSOS},
$\lim_{i\rightarrow\infty}\bss^{(k_i,\star)}_\be$ is a global minimizer of \eqref{eq::RCPSO}.
Consequently, if the set of optimal solutions of \eqref{eq::RCPSO} 
is a singleton, then $\lim_{k\rightarrow\infty}\bss^{(k,\star)}_\be$ is the unique global minimizer. 
\end{enumerate}
\end{theorem}
\begin{proof}

(i). Let $\by^{\star}$ be a minimizer of \eqref{eq::RCPSO}
and $\bar{\by}$ be the Slater point given in Assumption \ref{assump3}.
Because $f$, $\Y$ are convex by Assumption \ref{assump1}(i), 
we can choose $0<t<1$ such that $\by'\coloneqq t \by^{\star}+(1-t)\bar{\by}\in\Y$ 
and $f(\by')\le f^{\star}+\varepsilon$ for an arbitrary $\varepsilon>0$. 
As $-P(\by, \bx)$ is PSD-SOS-convex in $\by$ for every $\bx\in\X$, it holds
\[
P(\by', \bx)\succeq t P(\by^{\star},\bx)+ (1-t) P(\bar{\by},\bx)\succ 0,
\]
for all $\bx\in\X$. 
Let $\bss'=(s'_{\ba})_{\ba\in\N^{\ell}_{{2\lceil k_{\by}/2\rceil}}}$ with
$s'_{\ba}=(\by')^{\ba}$. By Assumption \ref{assump2} and Theorem \ref{th::psatz},
there exists $k^{(\varepsilon)}\in\N$ such that $\bss'$ is feasible to 
\eqref{eq::cdsdpSOS} for all $k\ge k^{(\varepsilon)}$. Therefore, 
$f^{\dsdp}_k\le f^{\star}+\varepsilon$ for all $k\ge k^{(\varepsilon)}$.
As $\varepsilon>0$ is arbitrary, by Proposition \ref{prop::lagSOS} and weak duality, we have $f^{\star}\le f^{\psdp}_k\le f^{\dsdp}_k$, and so 
$f^{\dsdp}_k\searrow f^{\star}$ and $f^{\psdp}_k\searrow f^{\star}$ as $k\to\infty$.

(ii). 
Let $\bss^{\star}=(s^{\star}_{\ba})_{\ba\in\N^{\ell}_{2\lceil k_{\by}/2\rceil}}$ be such 
that $\lim_{i \to \infty} s^{(k_i,\star)}_{\ba}=s^{\star}_{\ba}$ 
for all $\ba$. As the feasible set of \eqref{eq::RCPSO} is closed, by Corollary \ref{cor::feasible},
$\bss^{\star}_\be$ is feasible to \eqref{eq::RCPSO}. Moreover, 
as $f(\by)$ is SOS-convex by Assumption \ref{assump1}(i), by (i) and Proposition \ref{prop::js}, it holds that $f^{\star}=\mH_{\bss^{\star}}(f)\ge f(\bss^{\star}_{\be})$,
which indicates that $\bss^{\star}_\be$ is a global minimizer of \eqref{eq::RCPSO}. 
\end{proof} 

\begin{remark}\label{rk::boundSOS}
Suppose that there exists a ball constraint $\sum_{i=1}^{\ell} y_i^2\le b^2$ for some $b\neq 0$ in 
the description of $\Y$. Let $t\in\N$ satisfy $t\ge \max\,\{\lceil \deg(\theta_j)/2\rceil,\ j=1,\ldots,s\}$. Let $\bss=(s_{\ba})_{\ba\in\N^{\ell}_{2t}}$ satisfy $\bss\in\ms_t(\Theta)$ and $\mH_{\bss}(1)=1$. 
Then, by \cite[Lemma 3]{CD2016}, it holds that 
\[
\Vert \bss\Vert\le \sqrt{\binom{\ell+t}{\ell}}\sum_{i=0}^t b^{2i}.
\]
In this case, the convergent subsequence $(\bss^{(k_i,\star)})_i$  of $(\bss^{(k,\star)})_k$ in 
Theorem \ref{th::asym_convergence_SOS} (ii) always exists. 
\end{remark}

The next theorem allows us to detect finite convergence of the moment-SOS hierarchy \eqref{eq::cpsdpSOS}--\eqref{eq::cdsdpSOS} and to extract an optimal solution whenever the FEC is satisfied.
\begin{theorem}\label{th::finite_convergence_SOS}
Let $(f_k^{\psdp}, \mS^{(k,\star)})$ and 
$\bss^{(k,\star)}$ be optimal solutions to \eqref{eq::cpsdpSOS} and \eqref{eq::cdsdpSOS},
respectively, for some $k\ge \lceil k_{\bx}/2\rceil$.
If the following FEC 
\begin{equation}\label{eq::FEC_SOS_convex}
\exists\lceil k_{\bx}/2\rceil\le t\le k\,\text{ s.t. }\,
\rank(M_{t}(\mS^{(k,\star)}))=\rank(M_{t-d_G}(\mS^{(k,\star)}))
\end{equation}
holds, then 
\begin{enumerate}
\item[\upshape (i)] $\mS^{(k,\star)}$ admits a representing measure $\Phi^{\star}=\sum_{i=1}^{r}
W_i\delta_{\bx^{(i)}}\in\mathfrak{M}^m_+(\X)$ for some points $\bx^{(1)}, \ldots, \bx^{(r)}\in\X$ and $W_1, \ldots, W_{r}\in\bS_+^m$;
\item[\upshape(ii)] $f^{\psdp}_k\le f^{\star}$ and the equality holds if Assumptions \ref{assump1} and \ref{assump3} are satisfied.
\end{enumerate}
Consequently, under the condition \eqref{eq::FEC_SOS_convex} and Assumptions \ref{assump1} and \ref{assump3}, if 
$f^{\psdp}_k=f^{\dsdp}_k$, then $f^{\psdp}_k=f^{\dsdp}_k=f^{\star}$ and 
\begin{enumerate}
\item[\upshape(iii)] $\bss^{(k,\star)}_{\be}$ is an optimal solution to \eqref{eq::RCPSO};
\item[\upshape(iv)] For any decomposition $W_i=\sum_{l=1}^{m_i} \bv^{(i,l)}(\bv^{(i,l)})^{\intercal}$, $\bv^{(i,l)}\in\RR^m$, $i=1,\ldots, r$, it holds that
\[P\left(\bss^{(k,\star)}_{\be}, \bx^{(i)}\right) \bv^{(i,l)}=0, 
\quad l=1,\ldots,m_i,i=1,\ldots,r.\]
\end{enumerate}
\end{theorem}
\begin{proof}
(i). It follows from Theorem \ref{th::FEC}.
    
(ii). As $(f_k^{\psdp}, \mS^{(k,\star)})$ is feasible to \eqref{eq::cpsdpSOS}, thanks to (i), 
it holds
\begin{equation}\label{eq::ineqsos}
f(\by)-f^{\psdp}_k-\mL_{\mS^{(k,\star)}}(P(\by, \bx))=
f(\by)-f^{\psdp}_k-\sum_{i=1}^r\tr{W_iP(\by, \bx^{(i)})}\ge 0,
\end{equation}
for all $\by\in\Y$. Let $\by^{\star}$ be a global minimizer of 
\eqref{eq::RCPSO}. Noting
$\sum_{i=1}^r\tr{W_iP(\by^{\star}, \bx^{(i)})}\ge 0$, then by \eqref{eq::ineqsos},
we have
\[
f^{\psdp}_k\le f(\by^{\star})-\sum_{i=1}^r\tr{W_iP(\by^{\star}, \bx^{(i)})}
\le f(\by^{\star})=f^{\star}.
\]
If Assumptions \ref{assump1} and \ref{assump3} hold, then $f^{\psdp}_k\ge f^{\star}$ by Proposition \ref{prop::lagSOS}.  
Hence, $f^{\psdp}_k=f^{\star}$. 

(iii). By Corollary \ref{cor::feasible}, 
$\bss^{(k,\star)}_{\be}$ is feasible to \eqref{eq::RCPSO}. 
Since $f(\by)$ is SOS-convex, by (ii) and Proposition \ref{prop::js}, 
it holds
\[
f(\bss^{(k,\star)}_{\be})\le \mH_{\bss^{(k,\star)}}(f)=f^{\dsdp}_k=f^{\star},
\]
which implies that $\bss^{(k,\star)}_{\be}$ is a global minimizer of \eqref{eq::RCPSO}.

(iv). By (ii), (iii) and \eqref{eq::ineqsos}, we deduce that 
$-\sum_{i=1}^r\tr{W_iP(\bss^{(k,\star)}_{\be}, \bx^{(i)})}\ge 0$.
As $P(\bss^{(k,\star)}_{\be}, \bx^{(i)})$ and $W_i$ are both PSD, it follows
\[
0=\sum_{i=1}^r\tr{W_iP(\bss^{(k,\star)}_{\be}, \bx^{(i)})}=
\sum_{i=1}^r\sum_{l=1}^{m_i} (\bv^{(i,l)})^{\intercal}P(\bss^{(k,\star)}_{\be}, \bx^{(i)})\bv^{(i,l)}.
\]
The PSDness of $P(\bss^{(k,\star)}_{\be}, \bx^{(i)})$ implies
$P(\bss^{(k,\star)}_{\be}, \bx^{(i)}) \bv^{(i,l)}=0$ for all $l=1,\ldots,m_i$, $i=1,\ldots,r$.
\end{proof}


\subsubsection{Conditions for strong duality}
Now we give conditions under which there is no duality gap between \eqref{eq::cpsdpSOS} and \eqref{eq::cdsdpSOS}.
\begin{assumption}\label{assump4}
{\rm
(i) $\theta_{1}(\by)=b^2-\sum_{i=1}^{\ell}y_i^2$ for some $b\ne0$;
(ii) $\X$ has non-empty interior.}
\end{assumption}

As a consequence of \cite[Corollary 3.6]{MLMW2022}, we have the following result.
\begin{lemma}\label{lem::lambda}
Suppose that Assumption \ref{assump4}(i) holds. Then for any $k\ge\max\,\{\lceil k_{\bx}/2\rceil,\lceil k_{\by}/2\rceil\}$ and $\mS\in\ms^m_k(G)$, there exists $\lambda\in\RR$ such that
\[
f(\by)-\lambda-\mL_{\mS}(P(\by, \bx))\in\QM^\circ_k(\Theta),
\]
where $\QM_k^{\circ}(\Theta)$ denotes the interior of $\QM_k(\Theta)$.
\end{lemma}

\begin{lemma}\label{lem::S0}
Suppose that Assumption \ref{assump4}(ii) holds. Then for each $k\ge d_G$, there exists $\mS^{\circ}\in\ms^m_k(G)$
such that $M_k(\mS^{\circ})\succ 0$ and $M_{k-d_G}(G\mS^{\circ})\succ 0$.
\end{lemma}
\begin{proof}
We only prove that there exists $\mS^{\circ}\in\ms^m_k(G)$
such that $M_{k-d_G}(G\mS^{\circ})\succ 0$. Similar arguments apply also to
$M_k(\mS^{\circ})$.
Suppose on the contrary that the conclusion is false.
Let $\Phi\in\mathfrak{M}^m_+(\X)$ be such that $\Phi=\diag(\phi,\ldots,\phi)$
where $\phi$ is the probability measure with uniform distribution on $\X$,
and $\mS^{\circ}=(S_{\ba})_{\ba\in\N^n_{2k}}$ where each
$S_{\ba}=\int_{\X}\bx^{\ba}\ud\Phi(\bx)$. Now fix a nonzero vector 
$\bv\in\RR^{mq\snd{k-d_G}}$  such that  ${\bv^\intercal} M_{k-d_G}
(G\mS^{\circ})\bv=0$. Let
\[
\Sigma(\bx)=(u_{k-d_G}(\bx)\otimes I_{mq})^{\intercal} \bv\bv^{\intercal} 
(u_{k-d_G}(\bx)\otimes I_{mq}).
\]
Then by Proposition \ref{prop::Lrepre}, it holds 
\[
\mL_{\mS^{\circ}}((\Sigma, G)_m)=\tr{\bv\bv^{\intercal}M_{k-d_G}(G\mS^{\circ})}=
{\bv^\intercal} M_{k-d_G}(G\mS^{\circ})\bv=0.
\]
For each $i=1,\ldots,mq$, let $\bv^{(i)}$ be the subvector of $\bv$ whose entries are
indexed by 
\[
i,\ mq+i,\ 2mq+i,\ \ldots, (\snd{k-d_G}-1)mq+i,
\]
and $T_i(\bx)=(\bv^{(i)})^{\intercal}u_{k-d_G}(\bx)\in\RR[\bx]$. Then,
\[
[T_1(\bx),\ldots, T_{mq}(\bx)]=\bv^{\intercal} 
(u_{k-d_G}(\bx)\otimes I_{mq})\quad\text{and}\quad
\Sigma(\bx)=T(\bx)^{\intercal}T(\bx).
\]
For each $j=1,\ldots,m$, let 
\[
H_j(\bx)=[T_{(j-1)q+1}(\bx), \ldots, T_{jq}(\bx)]\in\RR[\bx]^{q}.
\]
Then, 
\[
(\Sigma, G)_m=[H_i(\bx)^{\intercal} G(\bx) H_j(\bx)]_{i,j=1,\ldots,m}\quad
\]
and
\[
0=\mL_{\mS^{\circ}}((\Sigma, G)_m)
=\int_{\X}\sum_{j=1}^m H_j(\bx)^{\intercal}G(\bx) H_j(\bx)\ud\phi(\bx).
\]
As $H_j(\bx)^{\intercal}G(\bx) H_j(\bx)\ge 0$ for all $\bx\in\X$, we have 
\[
\int_{\X}H_j(\bx)^{\intercal}G(\bx) H_j(\bx)\ud\phi(\bx)=0,\quad\forall j=1,\ldots,m.
\]
Let $\mathcal{O}$ be an open and bounded subset of $\X$. Then, 
there exists $\lambda>0$ such that $G(\bx)\succeq \lambda I_m$ on $\mathcal{O}$,
and for each $j=1,\ldots,m$, 
\[
\begin{aligned}
0&=\int_{\X}H_j(\bx)^{\intercal}G(\bx) H_j(\bx)\ud\phi(\bx)\ge \int_{\mathcal{O}}H_j(\bx)^{\intercal}G(\bx) H_j(\bx)\ud\phi(\bx)\\
&\ge  \lambda\int_{\mathcal{O}}H_j(\bx)^{\intercal} H_j(\bx)\ud\phi(\bx)=\lambda\int_{\mathcal{O}} \sum_{i=1}^q T_{(j-1)q+i}(\bx)^2 \ud\phi(\bx).
\end{aligned}
\]
Since $\mathcal{O}$ is open, we have $T_i(\bx)\equiv 0$ for each $i=1,\ldots,mq$.
We then conclude $\bv=0$, yielding a contradiction.
\end{proof}

\begin{theorem}\label{th::dualgap}
Under Assumption \ref{assump4},  we have $f^{\psdp}_k=f^{\dsdp}_k$ for each $k\ge \lceil k_{\bx}/2\rceil$.
\end{theorem}
\begin{proof}
By Lemma \ref{lem::lambda} and \ref{lem::S0}, 
\eqref{eq::cpsdpSOS} is strictly feasible and the conclusion follows.
\end{proof}

\subsubsection{An example}
For all numerical examples in the sequel, we use {\sf Yalmip} \cite{YALMIP} to model SDPs and then rely on {\sf Mosek} \cite{mosek} to solve them\footnote{The script is available at
\url{https://github.com/wangjie212/PMOptimization}.}. 

\begin{example}\label{ex::1}
Consider the following instance of \eqref{eq::RCPSO}:
\begin{equation}\label{eq::ex1}
\begin{aligned}
f^{\star}\coloneqq\inf_{\by\in\RR^2}&\ f(\by)\quad
\rm{s.t.}&\ P(\by, \bx)\succeq 0,\ 
\forall \bx\in \X\coloneqq\{\bx\in\RR^2 \mid G(\bx)\succeq 0\},\\
\end{aligned}
\end{equation}
where 
\[
P(\by, \bx)=
\left[
\begin{array}{ccc}
   1-(x_1y_1-x_2y_2)^2 & & 2(x_2y_1+x_1y_2) \\
   2(x_2y_1+x_1y_2)   && 1
\end{array}
\right]
\]
and 
\[
G(\bx)=
\left[
\begin{array}{cccc}
   1-x_1 & x_2 & 0 & 0   \\
   x_2   & 1+x_1 & 0 & 0 \\
   0 & 0  & x_1^2+x_2^2-1 & 0 \\
   0 & 0  & 0 & x_1x_2 \\
\end{array}
\right].
\]
It is clear that 
\[
\X=\{\bx\in\RR^2 \mid x_1^2+x_2^2=1,\ x_1x_2\ge 0\}.
\]
Geometrically, the feasible region is constructed by rotating the shape in the 
$\by$-plane defined by $y_1^2+4y_2^2\le 1$ continuously
around the origin by $90^{\circ}$ clockwise and then taking the common area of 
these shapes in this process. The feasible set (the gray area) is displayed in Figure \ref{fig::ex1}.
\begin{figure}
\centering
\scalebox{0.4}{
\includegraphics[trim=80 190 80 200,clip]{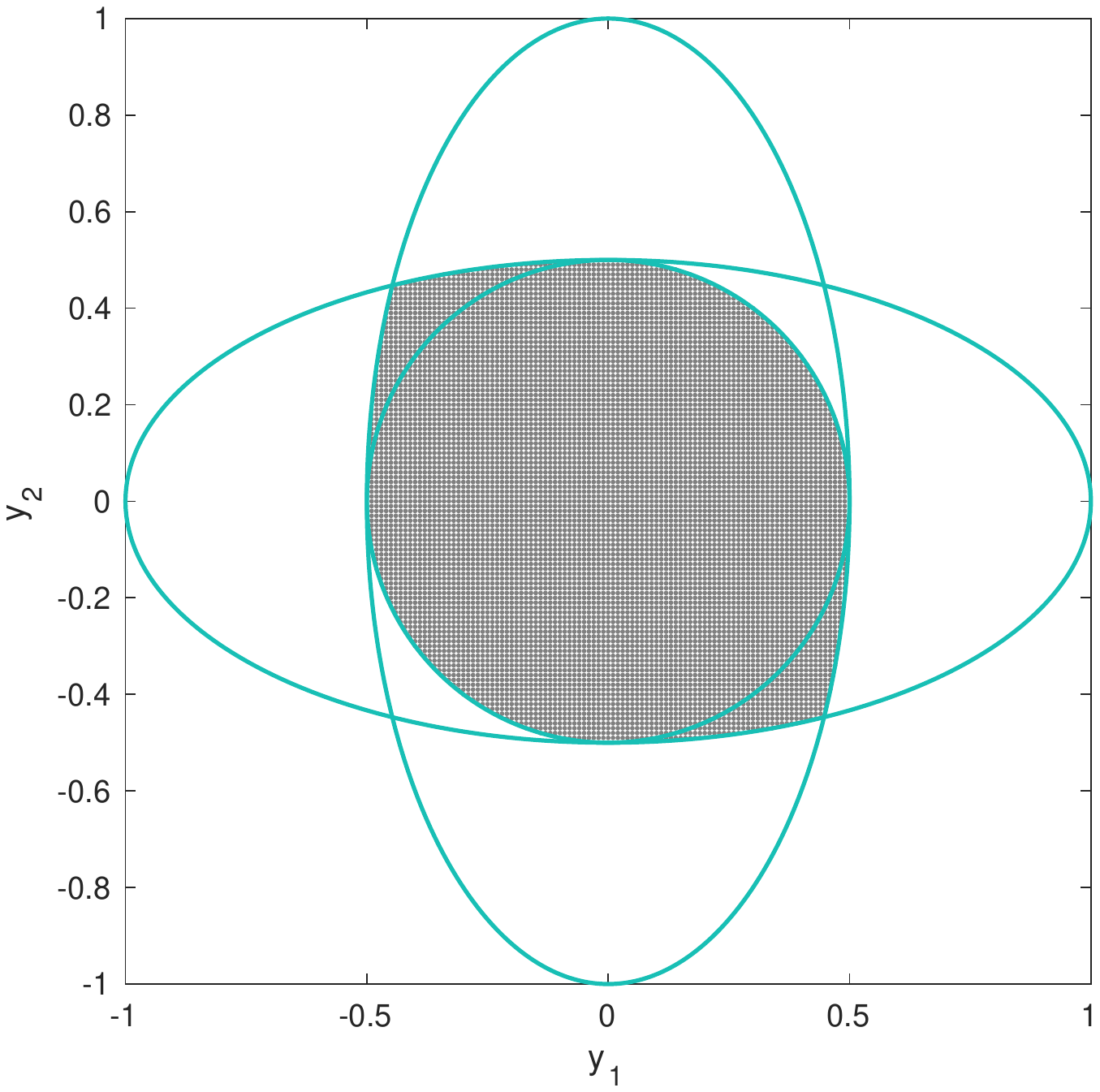}}
\caption{The feasible set (the gray area) of the problem \eqref{eq::ex1} 
in Example \ref{ex::1}\label{fig::ex1}. }
\end{figure}

By Corollary \ref{cor::psdSOSconcave}, it is easy to check that $-P(\by, \bx)$ is PSD-SOS-convex in $\by$ for every $\bx\in\X$. Consider the following objective functions $f_1(\by)=(y_1-1)^2+(y_2-1)^2$ and $f_2(\by)=(1+y_1)^2+(1-y_2)^2$, respectively. Clearly, both $f_1$ and $f_2$ are SOS-convex. 
The computational results are shown in Table \ref{tab::1}.

For $f_1$, the optimal solution is 
$(\sqrt{2}/4, \sqrt{2}/4)\approx (0.3536, 0.3536)$ and 
the optimal value is $2(\sqrt{2}/4-1)^2\approx 0.8358$. 
We solve the SDP relaxations \eqref{eq::cpsdpSOS} and \eqref{eq::cdsdpSOS} with $k=1, 2$. 
For $k=2$, the rank condition \eqref{eq::FEC_SOS_convex} is satisfied with $t=1$ and 
$f^{\psdp}_2=f^{\dsdp}_2=0.8358$. 
Then by Theorem~\ref{th::finite_convergence_SOS}, global optimality is certified and a minimizer $\bss^{(2,\star)}_{\be}=(0.3536, 0.3536)$ can be extracted. 

For $f_2$, the optimal solution is 
$(-\sqrt{5}/5, \sqrt{5}/5)\approx (-0.4472, 0.4472)$ and 
the optimal value is $2(\sqrt{5}/5-1)^2\approx 0.6111$. 
We solve the SDP relaxations \eqref{eq::cpsdpSOS} and \eqref{eq::cdsdpSOS} with $k=1, 2, 3$. 
As we can see, the global optimality is certified by the rank condition \eqref{eq::FEC_SOS_convex} at $k=3$ with $t=2$, even though it is already achieved at $k=2$.
A minimizer $\bss^{(3,\star)}_{\be}=(-0.4472, 0.4472)$ can be extracted.  
\begin{table}[htbp]
\caption{Computational results for Example \eqref{ex::1}.}\label{tab::1}
\renewcommand\arraystretch{1.2}
\centering
\resizebox{\linewidth}{!}{
\begin{tabular}{cccccccc}
\midrule[0.8pt]
\multirow{2}{*}{} &&\multicolumn{2}{c}{results for $f_1$}&
&\multicolumn{3}{c}{results for $f_2$} \\
		\cline{3-4} \cline{6-8}
&& $k=1$ & $k=2$& & $k=1$ & $k=2$ & $k=3$\\
\midrule[0.4pt]
		$f^{\psdp}_k$/time &&1.2497/0.61s& 0.8358/0.68s&&1.2499/0.62s&0.6111/0.67s&0.6111/1.00s \\
		$f^{\dsdp}_k$/time &&1.2496/0.52s& 0.8358/0.77s&&1.2497/0.44s&0.6111/0.77s &0.6111/2.08s\\
		FEC \eqref{eq::FEC_SOS_convex} && false & true && false & false & true\\
\midrule[0.8pt]
\end{tabular}}
\end{table}
\qed
\end{example}

\subsection{The linear case and the generalized matrix-valued moment problem}
If $\Y=\RR^{\ell}$, $f(\by)$ and $P(\bx,\by)$ are affine in $\by$, then \eqref{eq::RCPSO} becomes the robust polynomial semidefinite program \eqref{eq::LSIMP} which was extensively studied in \cite{SH2006}. In this case, we will see that the dual problem \eqref{eq::cdsdpSOS} recovers the matrix SOS relaxation for \eqref{eq::LSIMP} proposed in \cite{SH2006}. On the other hand, as a complement, the primal problem \eqref{eq::cpsdpSOS} allows us to detect finite convergence and extract optimal solutions.


\subsubsection{Robust polynomial semidefinite programming}
Consider the robust polynomial semidefinite programming problem which is a special case of \eqref{eq::RCPSO}:
\begin{equation}\label{eq::LSIMP}
\begin{aligned}
\tau^{\star}\coloneqq\inf_{\by\in\RR^{\ell}}&\ \bc^{\intercal}\by\quad
\text{s.t.}&\ \sum_{i=1}^\ell P_i(\bx) y_i-P_0(\bx)\succeq 0,\ 
\forall \bx\in \X\subseteq\RR^n,\\
\end{aligned}
 \tag{RPSDP}
\end{equation}
where $\bc=[c_1,\ldots,c_\ell]^{\intercal}\in\RR^{\ell}$ and $P_i(\bx)\in\bS[\bx]^{m}$, $i=0, 1,\ldots,\ell$.

Applying the conic reformulation \eqref{eq::cp}
to \eqref{eq::LSIMP} with $\Y=\RR^{\ell}$, we obtain
\begin{equation}\label{eq::lcp}
\left\{
   \begin{aligned}
     \sup_{\rho, \Phi}&\  \rho\\
      \text{s.t.}&\ \bc^{\intercal}\by-\rho-
      \sum_{i=1}^{\ell}\mL_{\Phi}(P_i(\bx))y_i+\mL_{\Phi}(P_0(\bx))
      \in \mathcal{P}(\RR^{\ell}),\\
      &\ \rho\in\RR,\ \Phi\in\mathfrak{M}^m_+(\X).
   \end{aligned} 
   \right.
\end{equation}
For any $\rho\in\RR$ and $\Phi\in\mathfrak{M}^m_+(\X)$, it holds
\[
\bc^{\intercal}\by-\rho- 
\sum_{i=1}^{\ell}\mL_{\Phi}(P_i(\bx))y_i+\mL_{\Phi}(P_0(\bx))
=\sum_{i=1}^{\ell}\left(c_i-\int_{\X}P_i(\bx)\ud\Phi(\bx)\right)y_i
-\rho+\int_{\X}P_0(\bx)\ud\Phi(\bx).
\]
Thus, if $(\rho, \Phi)$ is feasible to \eqref{eq::lcp}, then we necessarily have
\[
c_i=\int_{\X}P_i(\bx)\ud\Phi(\bx),\ i=1,\ldots,\ell,\quad \text{and}\quad
\rho\le \int_{\X}P_0(\bx)\ud\Phi(\bx),
\]
and \eqref{eq::lcp} can be rewritten as
\begin{equation}\label{eq::lcp2}
   \begin{aligned}
     \sup_{\Phi\in \mathfrak{M}^m_+(\X)}&\  
     \int_{\X}P_0(\bx)\ud\Phi(\bx)\quad
      \text{s.t.}&\ \int_{\X}P_i(\bx)\ud\Phi(\bx)=c_i,\ i=1,\ldots,\ell.
   \end{aligned} 
\end{equation}

\begin{remark}
Note that the generalized moment problem extensively studied by Lasserre in \cite{Lasserre2008} is a special case of \eqref{eq::lcp2} with $m=1$ and $G(\bx)$ being a diagonal matrix.
\end{remark}

The dual of \eqref{eq::lcp} reads as
\begin{equation}\label{eq::lcd}
\left\{
   \begin{aligned}
     \inf_{\mu}&\  \bc^{\intercal}\mH_{\mu}(\by)\\
      \text{s.t.}&\ \mu\in\mathfrak{m}_+(\RR^{\ell}),\ \mH_{\mu}(1)=1,\\
      &\ \sum_{i=1}^{\ell}P_i(\bx)\mH_{\mu}(y_i)-P_0(\bx)\in\mathcal{P}^m(\X). 
   \end{aligned} 
   \right.
\end{equation}
By identifying $\mH_{\mu}(\by)$ with $\by$, we see that \eqref{eq::lcd} is actually equivalent to \eqref{eq::LSIMP}.

For each $k\ge\lceil k_{\bx}/2\rceil$, replacing the cone $\mathfrak{M}^m_+(\X)$ in \eqref{eq::lcp2} by $\ms^m_k(G)$, 
we obtain the corresponding moment relaxation:
\begin{equation}\label{eq::lcpsdp}
\tau^{\psdp}_k\coloneqq\sup_{\mS\in\ms^m_k(G)}\  
     \mL_{\mS}(P_0(\bx))\quad
      \text{ s.t. }\quad\ \mL_{\mS}(P_i(\bx))=c_i,\ i=1,\ldots,\ell,
\end{equation}
whose dual is
\begin{equation}\label{eq::lcdsdp}
\tau^{\dsdp}_k\coloneqq\inf_{\by\in\RR^{\ell}}\ \bc^{\intercal}\by\quad
\text{s.t.}\quad\ \sum_{i=1}^{\ell} P_i(\bx) y_i-P_0(\bx)\in\QM^m_k(G).\ 
\end{equation}
Note that when $\X$ is compact, Assumption \ref{assump1} is satisfied for \eqref{eq::LSIMP}. 
Then, as direct consequences of Theorems \ref{th::asym_convergence_SOS} and 
\ref{th::finite_convergence_SOS}, we get the following
theorems.
\begin{theorem}\label{th::sym_conv_L}
Let Assumptions \ref{assump2}--\ref{assump3} hold.
Then, 
\begin{enumerate}
\item[\upshape (i)] $\tau^{\psdp}_k\searrow\tau^{\star}$ and $\tau^{\dsdp}_k\searrow\tau^{\star}$ as $k\to\infty$;
\item[\upshape (ii)] For any convergent subsequence $(\by^{(k_i,\star)})_i$ of 
$(\by^{(k,\star)})_k$ where $\by^{(k,\star)}$ is an optimal solution to \eqref{eq::lcdsdp}, $\lim_{i\to\infty} \by^{(k_i,\star)}$
is an optimal solution to \eqref{eq::LSIMP}.
\end{enumerate}
\end{theorem}

%
\begin{theorem}\label{th::finite_convergence_linear}
Let $\mS^{(k,\star)}$ and $\by^{(k,\star)}$ be optimal solutions to \eqref{eq::lcpsdp} and \eqref{eq::lcdsdp}, respectively, for some $k\ge \lceil k_{\bx}/2\rceil$. 
If the FEC \eqref{eq::FEC_SOS_convex} holds, then 
\begin{enumerate}
\item[\upshape (i)] $\mS^{(k,\star)}$ admits a 
representing measure $\Phi^{\star}=\sum_{i=1}^{r}
W_i\delta_{\bx^{(i)}}\in\mathfrak{M}^m_+(\X)$
for some points $\bx^{(1)}, \ldots, \bx^{(r)}\in\X$ and $W_1, \ldots, W_{r}\in\bS_+^m$;
\item[\upshape (ii)]$\tau^{\psdp}_k\le \tau^{\star}$ and the equality holds if $\X$ is compact and Assumption \ref{assump3} 
is satisfied.
\end{enumerate}
Consequently, under the condition \eqref{eq::FEC_SOS_convex}, the compactness of $\X$ and 
Assumption \ref{assump3}, if $\tau^{\psdp}_k=\tau^{\dsdp}_k$, then
$\tau^{\psdp}_k=\tau^{\dsdp}_k=\tau^{\star}$ and 
\begin{enumerate}
\item[\upshape (iii)] $\by^{(k,\star)}$ is an optimal solution to  \eqref{eq::LSIMP};
\item[\upshape (iv)] For any decomposition 
$W_i=\sum_{l=1}^{m_i} \bv^{(i,l)}(\bv^{(i,l)})^{\intercal}$, $\bv^{(i,l)}\in\RR^m$, 
$i=1,\ldots, r$, it holds that
\[
\left(\sum_{i=1}^{\ell}P_i(\bx^{(i)})y_i^{(k,\star)}-P_0(\bx^{(i)})\right)\bv^{(i,l)}=0, 
\quad l=1,\ldots,m_i,\ i=1,\ldots,r.
\]
\end{enumerate}
\end{theorem}
\begin{remark}
Note that, unlike \eqref{eq::cpsdpSOS}, there is no quadratic module constraint in \eqref{eq::lcpsdp}. So, 
if $\X$ has non-empty interior, then \eqref{eq::lcpsdp} is strictly feasible by Lemma \ref{lem::S0} 
and hence there is no duality gap between \eqref{eq::lcpsdp} and \eqref{eq::lcdsdp}. Besides, we can recover the moment relaxations for deterministic PMI optimization problems by Henrion and Lasserre \cite{HL2006} from \eqref{eq::lcpsdp}. 
\end{remark}

\subsubsection{An application: minimizing the smallest eigenvalue of a polynomial matrix}\label{sec:app}
Consider the problem of minimizing the smallest eigenvalue of a polynomial matrix:
\begin{equation}\label{eq::empo}
\lambda^{\star}\coloneqq\inf_{\bx\in\RR^n}\,\lambda_{\min\,}(F(\bx))\quad
\text{s.t.}\quad\bx\in\X\coloneqq\{\bx\in\RR^n \mid G(\bx)\succeq 0\},
\end{equation}
where $F(\bx)\in\bS[\bx]^{m}$, $\lambda_{\min\,}(F(\bx))$ denotes the 
smallest eigenvalue of $F(\bx)$  and $G(\bx)\in\bS[\bx]^{q}$. 
The motivations for studying this problem come from many different fields. 
For example, in the global optimization method $\ba$BB for general
constrained nonconvex problems, a convex relaxation of the original nonconvex problem
is constructed. To underestimate nonconvex terms of generic structure in the involved nonconvex functions, one needs to compute a parameter $\ba$ which amounts to minimizing the 
smallest eigenvalue of the corresponding Hessian matrix over a product of intervals. If the 
involved functions are polynomials, then the problem can be formulated as \eqref{eq::empo} 
\cite{AMF1995,MF1994}. For another example from optimal control, the stabilisability 
radius of a continuous-time system described by a state-space equation is defined as
the norm of the smallest perturbation that makes the system unstabilisable. 
To compute such a radius, one needs to minimize the  smallest eigenvalue of a bivariate
polynomial matrix over a half disc on a plane, which is a special case of the 
problem \eqref{eq::empo} \cite{DSS2009}.

Clearly, the problem \eqref{eq::empo} is equivalent to 
\begin{equation}\label{eq::empo2}
\sup_{\lambda\in\RR}\,\lambda\quad \text{s.t.}\quad F(\bx)-\lambda I_{m}\succeq 0,\quad \forall\bx\in\X,
\end{equation}
which is a special case of \eqref{eq::LSIMP}. 

For each $k\ge\max\,\{\lceil \deg(F)/2\rceil, \lceil\deg(G)/2\rceil\}$, the $k$-th order primal moment relaxation of \eqref{eq::empo} reads as
\begin{equation}\label{eq::eigpsdp}
     \lambda^{\psdp}_k\coloneqq\inf_{\mS\in\ms^m_k(G)}\,  
     \mL_{\mS}(F(\bx))\quad \text{s.t.}\quad\mL_{\mS}(I_m)=1,
\end{equation}
with dual
\begin{equation}\label{eq::eigdsdp}
\lambda^{\dsdp}_k\coloneqq\sup_{\lambda\in\RR}\,\lambda\quad \text{s.t.}\quad F(\bx)-\lambda I_{m}\in\QM^m_k(G).
\end{equation}
The dual problem \eqref{eq::eigdsdp} was studied in \cite{SL2023} where $\X$ is set to the $n$-dimensional boolean hypercube $\{0, 1\}^n$.

Clearly, the Slater condition holds for \eqref{eq::empo2} whenever $\X$ is compact.
Then, from Theorems \ref{th::sym_conv_L}--\ref{th::finite_convergence_linear},
we deduce the following theorems.

\begin{theorem}\label{th::asym_con_eig}
Let Assumption \ref{assump2} hold.
Then $\lambda^{\psdp}_k\nearrow\lambda^{\star}$ and $\lambda^{\dsdp}_k\nearrow\lambda^{\star}$ as $k\to\infty$.
\end{theorem}   
\begin{theorem}\label{th::finite_con_eig}
Let $k\ge\max\,\{\lceil \deg(F)/2\rceil, \lceil\deg(G)/2\rceil\}$ and $\mS^{\star}_k$ be an optimal solution to \eqref{eq::eigpsdp}. 
If the FEC \eqref{eq::FEC_SOS_convex} holds, then 
\begin{enumerate}
\item[\upshape (i)] $\mS^{\star}_k$ admits a 
representing measure $\Phi^{\star}=\sum_{i=1}^{r}
W_i\delta_{\bx^{(i)}}\in\mathfrak{M}^m_+(\X)$
for some points $\bx^{(1)}, \ldots, \bx^{(r)}\in\X$ and 
$W_1, \ldots, W_{r}\in\bS_+^m$ with $\sum_{i=1}^r \tr{W_i}=1$ (because $\mL_{\mS}(I_m)=1$);
\item[\upshape (ii)] $\lambda^{\psdp}_k=\lambda^{\star}$;
\item[\upshape (iii)] For any decomposition 
$W_i=\sum_{l=1}^{m_i} \bv^{(i,l)}(\bv^{(i,l)})^{\intercal}$, $\bv^{(i,l)}\in\RR^m$, 
$i=1,\ldots, r$, it holds 
\[
F(\bx^{(i)}) \bv^{(i,l)}=\lambda^{\star}\bv^{(i,l)}, 
\quad l=1,\ldots,m_i,i=1,\ldots,r.
\]
That is, the smallest eigenvalue $\lambda^{\star}$ of $F(\bx)$ over $\X$
is attained at $\bx^{(i)}$ with $\bv^{(i,l)}$ being the corresponding eigenvectors.
\end{enumerate}
\end{theorem}

\begin{example}
Consider the following instance of \eqref{eq::empo} where $F(\bx)=Q\diag(f_1, f_2, f_3)Q^\intercal$ for some $f_1, f_2, f_3\in\RR[\bx]$ and $Q\in\RR^{3\times 3}$ with $Q^{\intercal}Q=I_3$. Specifically, we let $Q$ and 
$G(\bx)$ be 
\[
Q=
\left[
\begin{array}{rrr}
   \frac{1}{\sqrt{2}}  & -\frac{1}{\sqrt{3}} &  \frac{1}{\sqrt{6}}\\
   0  &  \frac{1}{\sqrt{3}}  & \frac{2}{\sqrt{6}}\\
   \frac{1}{\sqrt{2}}  & \frac{1}{\sqrt{3}}  &  -\frac{1}{\sqrt{6}}
\end{array}
\right],
\quad
 G(\bx)=
 \left[
 \begin{array}{cc}
    1-4x_1x_2 & x_1 \\
    x_1 & 4-x_1^2-x_2^2 \\
 \end{array}
 \right],
\]  
and 
\[
f_1(\bx)=-x_1^2-x_2^2,\ 
f_2(\bx)=-\frac{1}{4}(x_1+1)^2-\frac{1}{4}(x_2-1)^2,\
f_3(\bx)=-\frac{1}{4}(x_1-1)^2-\frac{1}{4}(x_2+1)^2.
\]
It is easy to check that 
$\lambda^{\star}=\inf_{\bx\in\X}\{f_1(\bx), f_2(\bx), f_3(\bx)\}=-4$ which
is achieved by $f_1$ at $(0, \pm 2)$. The eigenvector space of 
$F(0,\pm 2)$ associated with the eigenvalue $-4$ consists of all vectors of the form
$[c, 0, c]^{\intercal}$, $c\in\RR$.

Solving the moment relaxation \eqref{eq::eigpsdp} with $k=2$, we get $\lambda^{\psdp}_2=-4.0000$ and 
the rank condition \eqref{eq::FEC_SOS_convex} is satisfied with 
$\rank(M_2(\mS^{\star}_2))=\rank(M_1(\mS^{\star}_2))=2$.
By Theorem \ref{th::finite_con_eig}, global optimality is reached. 
Using the procedure described in Section \ref{sec::recover}, we recover the representing
measure $\Phi^{\star}=\sum_{i=1}^{2} W_i\delta_{\bx^{(i)}}$ of $\mS^{\star}_2$ with
    \[
    \bx^{(1)}= (-0.0000, -2.0000),
\quad
\bx^{(2)}= (0.0000, 2.0000),
    \]
and
\[
W_1= \left[
\begin{array}{rrr}
    0.2505 &  -0.0000 &   0.2505\\  
   -0.0000 &  -0.0000 &   0.0000\\
    0.2505 &   0.0000 &   0.2505\\
\end{array}
\right],\ 
W_2= \left[
\begin{array}{rrr}
    0.2495 &   0.0000 &   0.2495\\  
    0.0000 &   0.0000 &  -0.0000\\
    0.2495 &  -0.0000 &   0.2495\\
\end{array}
\right].
\]
Both $W_1$ and $W_2$ have the decomposition:
$\gamma_i^2[1,\ 0,\ 1]^\intercal[1,\ 0,\ 1]$ for some constants $\gamma_1, \gamma_2\in\RR$.
Then, by Theorem \ref{th::finite_con_eig}, $[1,\ 0,\ 1]^\intercal$
is an eigenvector of $F(\bx^{(i)})$, $i=1, 2$, associated with the eigenvalue $-4.0000$.

\qed
\end{example}

\section{Extensions to the general convex and non-convex settings}\label{sec::extension}
In this section, we extend the results of Section \ref{sec::CRSDP} to the general convexity and non-convexity settings.
The following assumptions are obtained from Assumption \ref{assump1} by replacing ``SOS-convex" with ``convex".
\begin{assumption}\label{assump5}
{\rm
(i) $f(\by),-\theta_1(\by),\ldots,-\theta_s(\by)$ are convex;
(ii) $-P(\by,\bx)$ is PSD-convex in $\by$ for all $\bx\in\X$;
(iii) $\X$ is compact.
}
\end{assumption}

For each $k\ge\max\,\{\lceil k_{\bx}/2\rceil,\lceil k_{\by}/2\rceil\}$, by replacing the cones $\mathcal{P}(\Y)$ and $\mathfrak{M}^m_+(\X)$ in \eqref{eq::cp} 
with $\QM_k(\Theta)$ and $\ms^m_k(G)$ respectively, we obtain the following SDP:
\begin{equation}\label{eq::cpsdp}
\left\{
\begin{aligned}
     f^{\psdp}_k\coloneqq\sup_{\rho, \mS}&\  \rho\\
      \text{s.t.}&\ f(\by)-\rho-\mL_{\mS}(P(\by, \bx))\in \QM_k(\Theta),\\
      &\ \rho\in\RR,\ \mS\in\ms^m_k(G),
   \end{aligned} 
   \right.
\end{equation}
whose dual reads as
\begin{equation}\label{eq::cdsdp}
\left\{
   \begin{aligned}
     f^{\dsdp}_k\coloneqq\inf_{\bss}\,&\  \mH_{\bss}(f)\\
      \text{s.t.}&\ \bss\in\ms_k(\Theta),\ \mH_{\bss}(1)=1,\\
      &\ \mH_{\bss}(P(\by, \bx))\in\QM^m_k(G).
   \end{aligned} 
   \right.
\end{equation}



\begin{lemma}[Jensen's inequality for PSD-convexity]\label{lem::Jensen}
If a polynomial matrix $Q(\by)\in\bS[\by]^{m}$ is PSD-convex 
in $\by$ and a sequence $\bss=(s_{\ba})_{\ba\in\N^{\ell}}\subseteq\RR$ 
admits a representing probability measure, then 
$\mH_{\bss}(Q(\by)) \succeq Q(\bss_{\be})$. 
\end{lemma}
\begin{proof}
By the PSD-convexity of $Q(\by)$, for any $\bv\in\RR^m$, $\bv^{\intercal}Q(\by)\bv$ is convex in $\by$. 
Since $\bss$ admits a representing probability measure,
by Jensen's inequality for convex functions, it holds
$\bv^{\intercal}Q(\bss_\be)\bv\le \mH_{\bss}(\bv^{\intercal}Q(\by)\bv)
=\bv^{\intercal}\mH_{\bss}(Q(\by))\bv$
for any $\bv\in\RR^m$. Hence, $\mH_{\bss}(Q(\by))\succeq  Q(\bss_{\be}).$
\end{proof}

We have the following theorem whose proof can be found in Appendix \ref{app3}.
\begin{theorem}\label{th::asym_conv}
Under Assumptions \ref{assump2}, \ref{assump3}, \ref{assump4}(i) and \ref{assump5},
the following are true:
\begin{enumerate}
\item[\upshape (i)] $\lim_{k\rightarrow\infty}f^{\psdp}_k=\lim_{k\rightarrow\infty}f^{\dsdp}_k=f^{\star}$;
\item[\upshape (ii)] For any  convergent subsequence $(\bss^{(k_i,\star)}_\be)_i$ (always exists) of $(\bss^{(k,\star)}_\be)_k$ 
where each $\bss^{(k,\star)}$ is a minimizer of \eqref{eq::cdsdp}, $\lim_{i\rightarrow\infty}\bss^{(k_i,\star)}_\be$ 
is a global minimizer of \eqref{eq::RCPSO}. Consequently, if the optimal solution set of \eqref{eq::RCPSO} is a singleton, 
then $\lim_{k\rightarrow\infty}\bss^{(k,\star)}_\be$ is the unique global minimizer.
\end{enumerate}
\end{theorem}
\begin{remark}\label{rk::nonmonotonical}
Since $\QM_k(\Theta)$ is an inner approximation for $\mathcal{P}(\Y)$ while $\ms^m_k(G)$ is an outer approximation for $\mathfrak{M}^m_+(\X)$, the convergence of $(f^{\psdp}_k)_k$ and $(f^{\dsdp}_k)_k$ to $f^{\star}$ might not be monotonical (cf. Proposition \ref{prop::lagSOS}).
\end{remark}

As in the SOS-convex setting, finite convergence of \eqref{eq::cpsdp}--\eqref{eq::cdsdp} can be detected via certain FECs. 
Let $d_{\Theta}\coloneqq\max\,\{\lceil\deg(\theta_i)/2\rceil, i=1,\ldots, s\}$.
\begin{theorem}\label{th::finite_convergence}
Let $(\rho^{\star}_k, \mS^{(k,\star)})$ and $\bss^{(k,\star)}$ be optimal solutions to \eqref{eq::cpsdp} and \eqref{eq::cdsdp}, respectively, for some $k\ge\max\,\{\lceil k_{\bx}/2\rceil, \lceil k_{\by}/2\rceil\}$. 
Consider the following FECs
\begin{align}
&\exists\lceil k_{\bx}/2\rceil\le t_1\le k \text{ s.t. }\rank(M_{t_1}(\mS^{(k,\star)}))=\rank(M_{t_1-d_G}(\mS^{(k,\star)})),\label{eq::FEC_convex1}\\
&\exists\lceil k_{\by}/2\rceil\le t_2\le k\text{ s.t. }\rank(M_{t_2}(\bss^{(k,\star)}))=\rank(M_{t_2-d_{\Theta}}(\bss^{(k,\star)})).
\label{eq::FEC_convex2}
\end{align}
If the condition \eqref{eq::FEC_convex1} holds, then 
 \begin{enumerate}
    \item[\upshape (i)] $\mS^{(k,\star)}$ admits a 
    representing measure $\Phi^{\star}=\sum_{i=1}^{r_1}
    W_i\delta_{\bx^{(i)}}\in\mathfrak{M}^m_+(\X)$
    for some points $\bx^{(1)}, \ldots, \bx^{(r_1)}\in\X$,
    $W_1, \ldots, W_{r_1}\in\bS_+^m$ and $f^{\psdp}_k\le f^{\star}$.
\end{enumerate}
Under Assumption \ref{assump5},  if the condition \eqref{eq::FEC_convex2} holds, then
\begin{enumerate}
\item[\upshape (ii)] $\bss^{(k,\star)}$ admits a 
representing probability measure $\mu^{\star}=\sum_{j=1}^{r_2}
\lambda_j\delta_{\by^{(j)}}\in\mathfrak{m}_+(\Y)$ for some $\by^{(1)}, \ldots,
\by^{(r_2)}\in\Y$ and positive real numbers $\lambda_1, \ldots, \lambda_{r_2}$,
and $f^{\dsdp}_k \ge f^{\star}$.
\end{enumerate}
Consequently, under the conditions \eqref{eq::FEC_convex1}, \eqref{eq::FEC_convex2} 
and Assumption \ref{assump5}, if $f^{\psdp}_k=f^{\dsdp}_k$, then 
\begin{enumerate}
\item[\upshape (iii)]  $f^{\psdp}_k=f^{\dsdp}_k=f^{\star}$;
\item[\upshape(iv)] $\sum_{j=1}^{r_2}\lambda_j\by^{(j)}$ is an optimal solution to  \eqref{eq::RCPSO};
\item[\upshape(v)] For any decomposition  $W_i=\sum_{l=1}^{m_i} \bv^{(i,l)}(\bv^{(i,l)})^{\intercal}$, 
$\bv^{(i,l)}\in\RR^m$, $i=1,\ldots, r_1$, it holds that
\[
P\left(\sum_{j=1}^{r_2}\lambda_j\by^{(j)}, \bx^{(i)}\right) \bv^{(i,l)}=0, 
\quad l=1,\ldots,m_i,i=1,\ldots,r_1.
\]
\end{enumerate}
\end{theorem}
\begin{proof}
(i). See the proof of Theorem \ref{th::finite_convergence_SOS} (i).

(ii).  The first part follows from Theorem \ref{th::FEC}. 
For the second part, suppose that Assumption~\ref{assump5} holds.
As $\bss^{(k,\star)}$ is feasible to \eqref{eq::cdsdp}, 
by Lemma \ref{lem::Jensen},
\[
P\left(\sum_{j=1}^{r_2}\lambda_j\by^{(j)},\bx\right)\succeq 
\sum_{j=1}^{r_2}\lambda_j(P(\by^{(j)}, \bx))=\mH_{\bss^{(k,\star)}}(P(\by, \bx))
\succeq 0,
\]
for all $\bx\in\X$. So, by Jensen's inequality for $-\theta_i$'s, $\sum_{j=1}^{r_2}\lambda_j\by^{(j)}$ is feasible to 
\eqref{eq::RCPSO}. Then, as $f(\by)$ is convex, by Jensen's inequality again,
\begin{equation}\label{eq::minimizer}
f^{\star}\le f\left(\sum_{j=1}^{r_2}\lambda_j\by^{(j)}\right)\le 
\sum_{j=1}^{r_2}\lambda_j f(\by^{(j)})=\mH_{\bss^{(k,\star)}}(f)=f_k^{\dsdp}.
\end{equation}
Hence, it holds $f^{\dsdp}_k \ge f^{\star}$.

(iii). It follows from (i) and (ii).

(iv). It follows from (iii) and \eqref{eq::minimizer}.

(v). Similar to the proof of Theorem \ref{th::finite_convergence_SOS} (iv).
%
%
\end{proof}
\begin{remark}\label{rk::empty2}
Under Assumption \ref{assump4}, by Lemma \ref{lem::lambda} and \ref{lem::S0}, we can deduce that
\eqref{eq::cpsdp} is strictly feasible and hence $f^{\psdp}_k=f^{\dsdp}_k$.
\end{remark}

   Now we would like to remark on how and where the stronger 
   SOS-convexity conditions (Assumption \ref{assump1} (i) and (ii)) contribute to achieving 
   better results when relaxing the conic reformulations 
   \eqref{eq::cp} and \eqref{eq::cd}, compared to those achieved by using
   the general convexity conditions (Assumption \ref{assump5} (i) and (ii)). 
   Under Assumption \ref{assump1} and \ref{assump3}, the condition \eqref{eq::LQ} holds true as 
   proven in Proposition \ref{prop::lagSOS}.
   Consequently, replacing $\mathcal{P}(\Y)$ by the quadratic module 
   $\QM_{\lceil k_{\by}/2\rceil}(\Theta)$ of fixed order in \eqref{eq::cp}
   does not change its optimal value. 
   Hence, substituting the outer approximations $\ms^m_k(G)$, $k\in\N$, for $\mathfrak{M}^m_+(\X)$ in 
   \eqref{eq::cpsdpSOS} provides monotonically non-increasing 
   upper bounds on $f^{\star}$. Moreover, the SOS-convexity validates an extension of Jensen’s inequality
   (Proposition \ref{prop::js} and \cite[Theorem 2.6]{LasserreConvex}), which can be used to prove that
   a feasible point of \eqref{eq::cdsdpSOS} is also feasible to \eqref{eq::RCPSO} (Corollary \ref{cor::feasible}) and 
   an optimal solution to \eqref{eq::RCPSO} can be retrieved from \eqref{eq::cdsdpSOS} provided that the FEC 
   \eqref{eq::FEC_SOS_convex} holds (Theorem \ref{th::finite_convergence_SOS} (iii)). 
   In contrast, under the general convexity conditions, \eqref{eq::LQ} does not necessarily hold and we have to replace $\mathcal{P}(\Y)$ by $\QM_k(\Theta)$ with $k\ge\lceil k_{\by}/2\rceil$ in \eqref{eq::cpsdp}. Then, the convergence of $(f^{\psdp}_k)_k$ and 
   $(f^{\dsdp}_k)_k$ to $f^{\star}$  might not be monotonical as explained 
   in Remark~\ref{rk::nonmonotonical}. Additionally, to extract optimal solutions to \eqref{eq::RCPSO} from
   \eqref{eq::cdsdp} under general convexity conditions, 
  we require the FEC \eqref{eq::FEC_convex2},  besides the FEC \eqref{eq::FEC_convex1}, 
   to ensure Theorem \ref{th::finite_convergence} (ii) to hold. 

\begin{example}\label{ex::convex}
Consider the following instance of \eqref{eq::RCPSO}:
\[
f^{\star}\coloneqq\inf_{\by\in\RR^2}\, f(\by)\quad
{\text\rm s.t.}\quad P(\by, \bx)\succeq 0,\ 
\forall \bx\in \X\coloneqq\{\bx\in\RR^2 \mid G(\bx)\succeq 0\},
\]
where $P(\by, \bx)$ and $G(\bx)$ are defined as in Example \ref{ex::1}. 
The following polynomial (\cite[(5.2)]{APgap}) is convex but not SOS-convex:
\begin{equation}\label{eq::nonSOSconvex}
\begin{aligned}
h(\by)=\,&89-363y_1^4y_2+\frac{51531}{64}y_2^6-\frac{9005}{4}y_2^5+
\frac{49171}{16}y_2^4+721y_1^2-2060y_2^3-14y_1^3\\
&+\frac{3817}{4}y_2^2+363y_1^4-9y_1^5+77y_1^6+316y_1y_2+49y_1y_2^3-2550y_1^2y_2-968y_1y_2^2\\
&+1710y_1y_2^4+794y_1^3y_2+\frac{7269}{2}y_1^2y_2^2-\frac{301}{2}y_1^5y_2+\frac{2143}{4}y_1^4y_2^2+\frac{1671}{2}y_1^3y_2^3\\
&+\frac{14901}{16}y_1^2y_2^4-\frac{1399}{2}y_1y_2^5-\frac{3825}{2}y_1^3y_2^2-\frac{4041}{2}y_1^2y_2^3-364y_2+48y_1.
\end{aligned}
\end{equation}
Let $f(\by)=h(y_1-1,y_2-1)/10000$ which is again convex but not SOS-convex. 
We have $k_{\by}=6$ and $k_{\bx}=2$. 

Solving the SDPs \eqref{eq::cpsdp} and \eqref{eq::cdsdp} with $k=3$, we get $f^{\psdp}_3=f^{\dsdp}_3=0.4504$ and the rank conditions \eqref{eq::FEC_convex1} and \eqref{eq::FEC_convex2} are satisfied with $t_1=1$ and $t_2=3$. 
By Theorem \ref{th::finite_convergence},
global optimality is reached and so $f^{\star}\approx0.4504$. The extracted minimizer is $\by^{\star}=(0.2711, 0.4201)$.
\end{example}


Before ending this section, let us consider the most general case -- the problem \eqref{eq::RCPSO} without any convexity assumption. 
By Theorem \ref{th::finite_convergence} (i), if the FEC \eqref{eq::FEC_convex1} holds, 
then we obtain a lower bound $f^{\psdp}_k\le f^{\star}$. Moreover, we have the following theorem.
\begin{theorem}\label{th::finite_convergence_nonconvex}
Let $\bss^{(k,\star)}$ be a minimizer of \eqref{eq::cdsdp} for some 
$k\ge\max\,\{\lceil k_{\bx}/2\rceil, \lceil k_{\by}/2\rceil\}$.
Then,
\begin{enumerate}
\item[\upshape (i)] If $\rank(M_{\lceil k_{\by}/2\rceil}(\bss^{(k,\star)}))=1$, then $f^{\dsdp}_k\ge f^{\star}$ and 
the truncation sequence $\hat{\bss}^{(k,\star)}:=(\bss_{\ba}^{(k,\star)})_{|\ba|\le 2\lceil k_{\by}/2\rceil}$ 
of $\bss^{(k,\star)}$ admits a Dirac representing measure $\delta_{\by^{\star}}$ for some $\by^{\star}\in\Y$;
\item[\upshape (ii)] If the condition \eqref{eq::FEC_convex1} holds and $\rank(M_{\lceil k_{\by}/2\rceil}(\bss^{(k,\star)}))=1$, 
and moreover, $f^{\psdp}_k=f^{\dsdp}_k$, then $f^{\psdp}_k=f^{\dsdp}_k=f^{\star}$ and 
$\by^{\star}$ is a mininizer of \eqref{eq::RCPSO}.
\end{enumerate}
\end{theorem}
\begin{proof}
(i).  As $\rank(M_{\lceil k_{\by}/2\rceil}(\bss^{(k,\star)}))=1$, 
$\hat{\bss}^{(k,\star)}$ admits a Dirac representing measure $\delta_{\by^{(\star)}}$ for some $\by^{\star}\in\Y$.
Then, as $\bss^{(k,\star)}$ is feasible to \eqref{eq::cdsdp}, it holds that
\begin{equation}\label{eq::hats}
P(\by^{\star},\bx)=\mH_{\hat{\bss}^{(k,\star)}}(P(\by,\bx))=\mH_{\bss^{(k,\star)}}(P(\by,\bx))\succeq 0.
\end{equation}
So $\by^{\star}$ is feasible to \eqref{eq::RCPSO}
and hence $f^{\dsdp}_k=\mH_{\bss^{(k,\star)}}(f)=\mH_{\hat{\bss}^{(k,\star)}}(f)=f(\by^{\star})\ge f^{\star}$. 

(ii). It follows from Theorem \ref{th::finite_convergence} (i) and (i).
\end{proof}

\begin{example}\label{ex::nonconvex} Consider Example \ref{ex::1} where we let $f(\by)=(y_1+1)^2+(y_2-1)^2$ and 
modify the matrix $P(\by, \bx)$ to 
\[
P(\by, \bx)=
\left[
\begin{array}{ccc}
   1+(x_1y_1-x_2y_2)^2 & & 2(x_2y_1+x_1y_2) \\
   2(x_2y_1+x_1y_2)   && 1
\end{array}
\right].
\]
Geometrically, the feasible region is constructed by rotating the shape in the 
$\by$-plane defined by $-y_1^2+4y_2^2\le 1$ continuously
around the origin by $90^{\circ}$ clockwise and then taking the common area of 
these shapes in this process. Hence, $-P(\by, \bx)$ is {\itshape not} PSD-convex in $\by$ for every $\bx\in\X$.
It is easy to check that $f^{\star}=2\left(1-1/\sqrt{3}\right)^2\approx 0.3573$
with a unique minimizer $\left(-1/\sqrt{3}, 1/\sqrt{3}\right)\approx (-0.5774, 0.5774)$.

Solving the SDPs \eqref{eq::cpsdp} and \eqref{eq::cdsdp} with $k=3$, we get $f^{\psdp}_3=f^{\dsdp}_3=0.3573$.
As the rank condition \eqref{eq::FEC_convex1} is satisfied with $t=2$ and $\rank(M_{\lceil k_{\by}/2\rceil}(\bss^{(3,\star)}))=1$,
global optimality is reached by Theorem \ref{th::finite_convergence_nonconvex} (ii) 
and moreover, a minimizer $\by^{\star}=(-0.5773, 0.5774)$ is extracted.  \qed
\end{example}

\section{Conclusions}\label{sec::conclusions}
We have proposed a moment-SOS hierarchy for the robust PMI optimization problems with SOS-convexity 
for which asymptotic convergence is established and the FEC is used to detect global optimality. 
Extensions to the general convexity and non-convexity settings are also provided.
As this work generalizes most of the nice features of the moment-SOS hierarchy from the scalar polynomial 
optimization to the robust PMI optimization, we would expect to stimulate more applications of 
robust PMI optimization in different fields (e.g., robust optimization, control theory). 
For the scalar polynomial optimization, various algebraic structures 
(e.g., symmetry, sparsity \cite{waki2006sums,tssos1,wang2022cs}) 
can be exploited to derive a structured moment-SOS hierarchy with lower computational complexity. 
A recent work on exploiting sparsity for PMIs could be found in \cite{zheng2023sum}. 
As a line of further research, we intend to extend such techniques to the robust PMI optimization in future work.

\subsection*{Acknowledgments}
The authors are very grateful for the comments and suggestions of the anonymous referees 
which helped to improve the presentation.

\bibliographystyle{siamplain}
\bibliography{pmo}

\newpage

\appendix
\section{Uniform PSD-SOS-convexity}\label{secA0}
\begin{definition}{\upshape \cite{N2011}}
A polynomial matrix $Q(\by)\in\bS[\by]^{m}$ is \emph{uniformly} 
PSD-SOS-convex if there exists a polynomial matrix
$F(\bv, \by)$ in $(\bv, \by)$ such that
\begin{equation}\label{eq::upsc}
\nabla_{\by\by}(\bv^{\intercal}Q(\by)\bv)=F(\bv, \by)^{\intercal}F(\bv, \by).
\end{equation}
\end{definition}
Clearly, if $Q(\by)\in\bS[\by]^{m}$ is uniformly PSD-SOS-convex, then it is PSD-SOS-convex. Moreover, checking the existence of $F(\bv, \by)$ in \eqref{eq::upsc} can be converted into an SDP feasibility problem. 
For $Q(\by)\in\bS[\by]^{m}$, we write $Q(\by)=\sum_{\ba\in\supp{Q}} Q_{\ba}\by^{\ba}$, where $Q_{\ba}$ is the coefficient matrix of $\by^{\ba}$ in $Q(\by)$ and 
\[\supp{Q}\coloneqq\{\ba\in\N^{\ell}\mid\by^{\ba} \text{ appears in some } Q_{ij}(\by)\}.\]
Then we have the following proposition.
\begin{proposition}\label{prop::upsdsosconvex}
A polynomial matrix 
$Q(\by)=\sum_{\ba\in\supp{Q}}Q_{\ba}\by^{\ba}\in\bS[\by]^{m}$ is uniformly PSD-SOS-convex if 
$\sum_{\ba\in\supp{Q}} Q_{\ba}\otimes \nabla_{\by\by} \by^{\ba}$
is an SOS matrix.
\end{proposition}
\begin{proof}
Observe that 
\[
\begin{aligned}
\nabla_{\by\by}(\bv^{\intercal}Q(\by)\bv)
&=\sum_{i,j=1}^m \left(\sum_{\ba\in\supp{Q}} [Q_{\ba}]_{ij} \nabla_{\by\by} \by^{\ba}
\right)v_iv_j\\
&=(\bv\otimes I_{\ell})^{\intercal}\left(\sum_{\ba\in\supp{Q}}Q_{\ba}\otimes
\nabla_{\by\by} \by^{\ba}\right)(\bv\otimes I_{\ell}).
\end{aligned}
\]
If there exists a polynomial matrix $T(\by)$ such that
\[
\left(\sum_{\ba\in\supp{Q}}Q_{\ba}\otimes \nabla_{\by\by} \by^{\ba}\right)=
T(\by)^{\intercal}T(\by),
\]
we then obtain \eqref{eq::upsc} by letting $F(\bv, \by)=T(\by) (\bv\otimes I_{\ell})$.
\end{proof}

\begin{corollary}\label{cor::psdSOSconcave}
A quadratic polynomial matrix 
\[Q(\by)=C+\sum_{i=1}^\ell
    L_iy_i+\sum_{i,j=1}^{\ell}Q_{ij}y_iy_j\in\bS[\by]^{m},\]
where $C, L_i, Q_{ij}\in\bS^m$ and 
$Q_{ij}=Q_{ji}$, is uniformly PSD-SOS-convex if the $m\ell\times m\ell$ matrix 
$[Q_{ij}]_{1\le i,j\le \ell}$ is PSD.
\end{corollary}
\begin{proof}
It is clear that $\sum_{\ba\in\supp{Q}} \nabla_{\by\by} \by^{\ba} \otimes Q_{\ba}=2[Q_{ij}]_{1\le i,j\le \ell}$ which implies that the matrix 
$\sum_{\ba\in\supp{Q}} Q_{\ba}\otimes \nabla_{\by\by} \by^{\ba}\in \bS^{\ell m \times \ell m}$ is PSD. Hence, $Q(\by)$ is uniformly PSD-SOS-convex by Proposition \ref{prop::upsdsosconvex}. 
\end{proof}

\section{An example for matrix-valued measure recovery}\label{secA1}
We construct a finitely atomic matrix-valued measure and then recover it from 
its moment matrix. Let $m=n=k=2$, $r=3$, 
    \[
    \bx^{(1)}=(0.3855, -0.2746),
    \quad
\bx^{(2)}= (-0.5863, 0.9648)
\quad
\bx^{(3)}= (0.1130, -0.8247),
    \]
and
\[
W_1= \left[
\begin{array}{rr}
 0.6731 &  -0.7569\\
   -0.7569 &   0.8512
\end{array}
\right],\ 
W_2= \left[
\begin{array}{rr}
 0.6399  &  0.5259\\
    0.5259  &  0.8048
\end{array}
\right],\
W_3= \left[
\begin{array}{rr}
0.0661 &  -0.2294\\
   -0.2294  &  0.7968
\end{array}
\right].
\]
Note that $\rank(W_1)=\rank(W_3)=1$ and $\rank(W_3)=2$. 
We let $\Phi=\sum_{i=1}^3 W_i\delta_{\bx^{(i)}}$. 
Denote by $\mS=(S_{\ba})_{\ba\in\N^2_4}$ and $M_2(\mS)$ the associated truncated 
moment sequence and moment matrix, respectively. 
Next, we recover $\bx^{(i)}$'s and $W_i$'s from $M_2(\mS)$ by the procedure described in 
Section \ref{sec::recover}.

We have $t=\rank(M_2(\mS))=\rank(M_1(\mS))=4$.
Compute the Cholesky decomposition 
$M_2(\mS)=\widetilde{V}\widetilde{V}^\intercal$ with 
$\widetilde{V}\in\RR^{12\times 4}$ and reduce matrix $\widetilde{V}$ to the column echelon form
\[
U=\left[
\begin{array}{rrrr}
    1.0000 &        0 &        0 &        0 \\       
         0 &   1.0000 &        0 &        0 \\       
         0 &        0 &   1.0000 &        0 \\       
         0 &        0 &        0 &   1.0000 \\       
    0.5775 &   0.3205 &  -0.6606 &   0.5467 \\       
   -1.2518 &  -0.8961 &  -2.1350 &  -3.1739 \\       
    0.3025 &   0.0680 &  -0.0703 &   0.1160 \\       
   -0.2657 &  -0.0103 &  -0.4532 &  -0.6038 \\       
   -0.3451 &  -0.0506 &   0.3762 &  -0.0862 \\       
    0.1975 &  -0.1126 &   0.3368 &   0.7727 \\       
    0.2682 &  -0.1303 &  -1.1301 &  -0.2222 \\       
    0.5088 &   0.8672 &   0.8678 &  -0.1086 \\
\end{array}
\right].
\] 
The rows of $U$ correspond to the monomials 
\[
v_2(\bx,\bw)=[w_1, w_2, x_1w_1, x_1w_2, x_2w_1, x_2w_2, x_1^2w_1, x_1^2w_2, x_1x_2w_1
x_1x_2w_2, x_2^2w_1, x_2^2w_2]^{\intercal}.
\]
From $U$, we can read the generating basis 
$b_2(\bx,\bw)=[w_1, w_2, x_1w_1, x_1w_2]$ which satisfies that $v_2(\bv,\bw)=Ub_2(\bx,\bw)$ holds 
at each pair $\bx^{(i)}$ and $\bw^{(i,j)}$. Moreover, we can get from $U$ the 
multiplication matricex of $x_1$ and $x_2$ with respect to $b_2(\bx,\bw)$:
\[
N_1=\left[
\begin{array}{rrrr}
         0 &        0 &   1.0000 &        0\\  
         0 &        0 &        0 &   1.0000\\
    0.3025 &   0.0680 &  -0.0703 &   0.1160\\
   -0.2657 &  -0.0103 &  -0.4532 &  -0.6038\\
\end{array}
\right],
\]
and 
\[
N_2=\left[
\begin{array}{rrrr}
    0.5775 &   0.3205&   -0.6606&    0.5467\\  
   -1.2518 &  -0.8961&   -2.1350&   -3.1739\\
   -0.3451 &  -0.0506&    0.3762&   -0.0862\\
    0.1975 &  -0.1126&    0.3368&    0.7727\\
\end{array}
\right].
\]
Build a random combination of multiplication matrices 
$N=0.3687N_1+ 0.6313N_2,$
and compute the ordered Schur decomposition $N=ATA^{\intercal}$ with
\[
A=\left[
\begin{array}{rrrr}
    0.2749&   -0.7785&   -0.3800&    0.4171\\  
   -0.9549&   -0.2787&   -0.1012&    0.0171\\
    0.0311&   -0.4340&    0.9003&   -0.0103\\
   -0.1079&    0.3577&    0.1866&    0.9086\\
\end{array}
\right].
\]
Compute the $4$ points in \eqref{eq::lpoints}:
\[
\begin{aligned}
\left[
\begin{array}{c}
      a_1^\intercal N_1 a_1 \\
    a_1^\intercal N_2 a_1
\end{array}
\right]
=
\left[
\begin{array}{r}
0.1130\\
   -0.8247
\end{array}
\right],\quad
&
\left[
\begin{array}{c}
      a_2^\intercal N_1 a_2 \\
    a_2^\intercal N_2 a_2
\end{array}
\right]
=
\left[
\begin{array}{r}
 0.3855\\
   -0.2746
\end{array}
\right],\\
\left[
\begin{array}{c}
      a_3^\intercal N_1 a_3 \\
    a_3^\intercal N_2 a_3
\end{array}
\right]
=
\left[
\begin{array}{r}
-0.5863\\
    0.9648
\end{array}
\right],\quad
&
\left[
\begin{array}{c}
      a_4^\intercal N_1 a_4 \\
    a_4^\intercal N_2 a_4
\end{array}
\right]
=
\left[
\begin{array}{r}
-0.5863\\
    0.9648
\end{array}
\right].
\end{aligned}
\]
As we can see, all points $\bx^{(1)}$, $\bx^{(2)}$ and $\bx^{(3)}$ have been recovered.
Among the above 4 points, $\bx^{(1)}$, $\bx^{(3)}$ appear one time and 
$\bx^{(3)}$ appears two times, which corresponds to the ranks of $W_1$, $W_2$ and $W_3$.
We compute the matrix $\Lambda$ and find $\mathcal{R}=\{1,2,3,4,5,6\}$ indexing the 
$mr=6$ independent rows in $\Lambda$. We have
\[
\Lambda_{\mathcal{R}}=\left[
\begin{array}{rrrrrr}
    1.0000 &        0 &   1.0000 &        0 &   1.0000 &        0\\  
         0 &   1.0000 &        0 &   1.0000 &        0 &   1.0000\\
    0.3855 &        0 &  -0.5863 &        0 &   0.1130 &        0\\
         0 &   0.3855 &        0 &  -0.5863 &        0 &   0.1130\\
   -0.2746 &        0 &   0.9648 &        0 &  -0.8247 &        0\\
         0 &  -0.2746 &        0 &   0.9648 &        0 &  -0.8247\\
\end{array}
\right],
\]
and 
\[
M_{\mathcal{R}}(\mS)=\left[
\begin{array}{rr}
    1.3791 &  -0.4605\\  
   -0.4605 &   2.4528\\
   -0.1082 &  -0.6261\\
   -0.6261 &  -0.0537\\
    0.3781 &   0.9044\\
    0.9044 &  -0.1145\\
\end{array}
\right].
\]
Now we compute
\[
\Lambda_{\mathcal{R}}^{-1}  M_{\mathcal{R}}(\mS)
=\left[
\begin{array}{rr}
    0.6731 &  -0.7569\\  
   -0.7569 &   0.8512\\
    0.6399 &   0.5259\\
    0.5259 &   0.8048\\
    0.0661 &  -0.2294\\
   -0.2294 &   0.7968\\
\end{array}
\right],
\]
which corresponds exactly to $[W_1, W_2, W_3]^\intercal$.





\section{Proof of Theorem \ref{th::asym_conv}}\label{app3}
\begin{proof}

(i). Fix an arbitrary $\varepsilon>0$. Under Assumptions \ref{assump2}, \ref{assump3} and 
\ref{assump5}, the proof of Theorem \ref{th::asym_convergence_SOS} (i) implies 
that there exists $k_1^{(\varepsilon)}\in\N$ 
such that $f^{\dsdp}_k\le f^{\star}+\varepsilon$ for all $k\ge k_1^{(\varepsilon)}$.

Let $\Phi^{\star}\in\mathfrak{M}^m_+(\X)$ be the finitely atomic matrix-valued 
measure in Proposition \ref{prop::sduality} and let $\mS^{\star}=
(S^{\star}_{\ba})_{\ba\in\N^n_{2k}}$ with each $S^{\star}_{\ba}=\int_{\X}\bx^{\ba}\ud\Phi^{\star}(\bx)$.
Then $\mS^{\star}\in\ms^m_k(G)$. Due to Proposition \ref{prop::sduality} and Remark \ref{rk::convex}, 
it holds
\begin{equation}\label{eq::fL}
f(\by)-(f^{\star}-\varepsilon)-\mL_{\mS^{\star}}(P(\by, \bx))>0, 
\end{equation}
for all $\by\in\Y$. As Assumption \ref{assump4}(i) holds, by Putinar's Positivstellens\"atz \cite{Putinar1993} 
(see also Theorem \ref{th::psatz}), there exists $k_2^{(\epsilon)}\in\N$, such that for all $k\ge k_2^{(\epsilon)}$, 
$(f^{\star}-\varepsilon, \mS^{\star})$ is feasible to \eqref{eq::cpsdp} and hence 
$f_k^{\psdp}\ge f^{\star}-\varepsilon$.

As $\varepsilon>0$ is arbitrary, by the weak duality, we have 
$\lim_{k\to\infty} f^{\psdp}_k=\lim_{k\to\infty} f^{\dsdp}_k=f^{\star}$.


(ii). For each $\ba\in\N^{\ell}$, define
\[
N(\ba)\coloneqq\sqrt{\binom{\ell+\left\lceil\frac{\vert \ba\vert}{2}\right\rceil}{\ell}}
\sum_{i=1}^{\left\lceil\frac{\vert \ba\vert}{2}\right\rceil} b^{2i}.
\]
By Assumption \ref{assump4}(i) and Remark \ref{rk::boundSOS}, $\vert s^{(k,\star)}_{\ba}\vert\le N(\ba)$ 
for all $k$ and 
$\ba\in\N^{\ell}_{2k}$. Hence, there always exists a convergent subsequence of $(\bss^{(k,\star)}_\be)_k$. Without loss of generality, we assume that the whole sequence $(\bss^{(k,\star)}_\be)_k$ converges. 
Complete each $\bss^{(k,\star)}$ with zeros to make it an 
infinite vector. Then,
\[
\left\{(s^{(k,\star)}_{\ba})_{\ba\in\N^{\ell}}\right\}_k
\subseteq\prod_{\ba\in\N^{\ell}}\left[-N(\ba), N(\ba)\right].
\]
By Tychonoff’s theorem, the product space 
$\prod_{\ba\in\N^{\ell}}\left[-N(\ba), N(\ba)\right]$ is compact in the 
product topology. Therefore, there exists a subsequence 
$(\bss^{(k_i,\star)})_i$ of $(\bss^{(k,\star)})_k$ 
and $\bss^{\star}=(s^{\star}_{\ba})_{\ba\in\N^{\ell}}$  such 
that $\lim_{i \to \infty} s^{(k_i,\star)}_{\ba}=s^{\star}_{\ba}$ 
for all $\ba\in\N^{\ell}$. By the pointwise convergence, we have 
(a) $\bss^{\star}\in\ms_k(\Theta)$ for all $k$;  (b) $\mH_{\bss^{\star}}(1)=1$; 
(c) $\mH_{\bss^{\star}}(P(\by, \bx))\succeq 0$ for all $x\in\X$.
By (a), (b), Putinar’s Positivstellens\"atz and Haviland’s theorem, $\bss^{\star}$
has a representing probability measure $\mu$ supported on $\Y$, i.e.,
$s^{\star}_{\ba}=\int_{\Y} \by^{\ba}\ud\mu(\by)$ for all $\ba\in\N^{\ell}$.
By Lemma \ref{lem::Jensen},
$P(\bss^{\star}_\be, \bx)\succeq \mH_{\bss^{\star}}(P(\by, \bx))\succeq 0$ 
and $\theta_i(\bss^{\star}_\be)\ge \mH_{\bss^{\star}}(\theta_i)\ge 0$, 
$i=1,\ldots,s$.
Hence $\bss^{\star}_\be$ is feasible to \eqref{eq::RCPSO}. Moreover, 
as $f(\by)$ is convex in $\by$, by (i) and the pointwise convergence, 
\[
f^{\star}=\mH_{\bss^{\star}}(f)\ge f(\bss^{\star}_{\be}),
\]
which indicates that $\bss^{\star}_\be$ is a minimizer of \eqref{eq::RCPSO}. 
\end{proof}


\end{document}

%% file: ex_shared.tex
\usepackage{lipsum}
\usepackage{amsfonts}
\usepackage{graphicx}
\usepackage{epstopdf}
\usepackage{algorithmic}
\usepackage{enumerate}
\usepackage{mathtools}
\usepackage{amssymb}
\usepackage{mathrsfs}
\usepackage{multirow}
\usepackage{booktabs}
\ifpdf
  \DeclareGraphicsExtensions{.eps,.pdf,.png,.jpg}
\else
  \DeclareGraphicsExtensions{.eps}
\fi

\usepackage{enumitem}
\setlist[enumerate]{leftmargin=.5in}
\setlist[itemize]{leftmargin=.5in}

\newsiamremark{remark}{Remark}
\newsiamremark{hypothesis}{Hypothesis}
\crefname{hypothesis}{Hypothesis}{Hypotheses}
\newsiamthm{claim}{Claim}
\newtheorem{assumption}{Assumption}
\newtheorem{example}[theorem]{Example}

\headers{A Moment-SOS Hierarchy for Robust PMI Optimization with SOS-Convexity}{F. Guo, and J. Wang}

\title{A Moment-SOS Hierarchy for Robust Polynomial Matrix Inequality Optimization with SOS-Convexity\thanks{Submitted to the editors DATE.
\funding{FG was supported by the NSFC under grant 11571350. JW was supported by the NSFC grant 12201618.}}}

\author{Feng Guo\thanks{School of Mathematical Sciences, Dalian  University of Technology
  (\email{fguo@dlut.edu.cn}).}
\and Jie Wang\thanks{Academy of Mathematics and Systems Science, CAS
  (\email{wangjie212@amss.ac.cn}, \url{https://wangjie212.github.io/jiewang/}).}}

\usepackage{amsopn}
\DeclareMathOperator{\diag}{diag}
